\documentclass[3p,times]{elsarticle}

\journal{Mathematics and Computers in Simulation}

\usepackage{amsmath}
\usepackage{mathrsfs}
\usepackage{graphicx}
\usepackage{url}
\usepackage{array}

\newdefinition{definition}{Definition}[section]
\newtheorem{theorem}{Theorem}[section]
\newtheorem{lemma}[theorem]{Lemma}
\newtheorem{proposition}[theorem]{Proposition}

\newtheorem{remark}{Remark}[section]

\newproof{proof}{Proof}
\numberwithin{equation}{section}

\newcommand{\mathbd}[1]{\boldsymbol{#1}}
\newcommand{\divv}{\mathrm{d}}
\newcommand{\diff}{\,\divv}

\renewcommand{\Im}{\operatorname{Im}}
\renewcommand{\pi}{\piup}
\DeclareMathOperator{\Order}{O}
\DeclareMathOperator{\rme}{e}
\DeclareMathOperator{\Si}{Si}
\newcommand{\domD}{\mathscr{D}}
\newcommand{\MC}{\mathbf{M}}
\newcommand{\Hinf}{\mathbf{H}^{\infty}}
\newcommand{\Ident}{\mathcal{I}}
\newcommand{\textRZ}{\text{\tiny{\rm{RZ}}}}
\newcommand{\textSE}{\text{\tiny{\rm{SE}}}}
\newcommand{\textDE}{\text{\tiny{\rm{DE}}}}
\newcommand{\Vol}{\mathcal{V}}
\newcommand{\VolSEn}{\mathcal{V}_N^{\textSE}}
\newcommand{\VolDEn}{\mathcal{V}_N^{\textDE}}
\newcommand{\SEt}{\psi^{\textSE}}
\newcommand{\DEt}{\psi^{\textDE}}
\newcommand{\SEtInv}{\phi^{\textSE}}
\newcommand{\DEtInv}{\phi^{\textDE}}
\newcommand{\SEtDiv}{\{\SEt\}'}
\newcommand{\DEtDiv}{\{\DEt\}'}
\newcommand{\ProjSE}{\mathcal{P}_N^{\textSE}}
\newcommand{\ProjRZ}{\mathcal{P}_N^{\textRZ}}
\newcommand{\ProjDE}{\mathcal{P}_N^{\textDE}}
\newcommand{\vRZn}{w_N^{\textRZ}}
\newcommand{\vSEn}{v_N^{\textSE}}
\newcommand{\vDEn}{v_N^{\textDE}}
\newcommand{\uSEn}{u_N^{\textSE}}
\newcommand{\uDEn}{u_N^{\textDE}}
\newcommand{\tSE}{t^{\textSE}}
\newcommand{\tDE}{t^{\textDE}}
\newcommand{\kSEn}{V_{n}^{\textSE}}
\newcommand{\gSEn}{\mathbd{g}_n^{\textSE}}
\newcommand{\kDEn}{V_{n}^{\textDE}}
\newcommand{\gDEn}{\mathbd{g}_n^{\textDE}}

\begin{document}

\begin{frontmatter}

\title{Relation between two Sinc-collocation methods for Volterra integral equations of the second kind and further improvement~\tnoteref{mytitlenote}}
\tnotetext[mytitlenote]{This work was supported in part by JSPS KAKENHI
Grant-in-Aid for Scientific Research (C), Grant Number JP23K03218.}

\author[HCU]{Tomoaki Okayama\corref{cor1}}
\cortext[cor1]{Corresponding author}
\affiliation[HCU]{organization={Hiroshima City University},
addressline={3-4-1, Ozuka-higashi, Asaminami-ku},
city={Hiroshima},
postcode={731-3194},
country={Japan}}
\ead{okayama@hiroshima-cu.ac.jp}

\begin{abstract}
Stenger and Rashidinia--Zarebnia independently proposed two
Sinc-collocation methods for Volterra integral equations of the second kind,
but the relation between the methods has not been clarified.
This study reformulates Stenger's method for general two-variable kernels
and rigorously establishes its applicability and convergence.
We prove that the approximate functions produced by the two methods are not
generally identical, although they coincide at every collocation point.
We also provide a rigorous convergence proof for the
Rashidinia--Zarebnia method and show that both methods attain the same
root-exponential convergence rate.  Because Stenger's method is simpler to
implement than the Rashidinia--Zarebnia method, we adopt it as
the basis for further improvement.  By replacing the tanh transformation
with the double-exponential transformation, we develop a new
Sinc-collocation method and prove its almost exponential convergence.
Numerical examples support the theoretical results and demonstrate the
favorable balance between accuracy and computational cost of the proposed
method.
\end{abstract}


\end{frontmatter}

\section{Introduction}
\label{sec:introduction}

This paper is concerned with numerical solutions via
Sinc numerical methods~\cite{stenger00:_summar,sugihara04:_recen}
for Volterra integral equations of the second kind of the form
\begin{equation}
u(t) - \int_a^t k(t,s) u(s)\diff s
= g(t),\quad a\leq t\leq b.
\label{eq:Volterra-int}
\end{equation}
Here, $k(t,s)$ and $g(t)$ are given continuous functions,
and $u(t)$ is the solution to be determined.
Equation~\eqref{eq:Volterra-int} is often expressed symbolically
as $(\Ident -\Vol) u = g$ by introducing the
Volterra integral operator $\Vol:C([a,b])\to C([a,b])$ defined by
\[
 \Vol[f](t)=\int_a^t k(t,s)f(s)\diff s.
\]
One of the powerful tools in Sinc numerical methods,
especially for the target equation~\eqref{eq:Volterra-int},
is the Sinc indefinite integration~\cite{haber93:_two,muhammad03:_doubl}.
It provides an approximation formula for an indefinite integral
of the following form:
\[
 \int_a^t F(s)\diff s \approx
\sum_{j=-N}^N F(s_j) \omega_j(t),
\]
where the weight $\omega_j$ depends on $t$, whereas
the sampling point $s_j$ is \emph{fixed} and independent of $t$,
even though the interval of the integral $(a,\,t)$ depends on $t$.
This is a distinctive feature because both $\omega_j$ and $s_j$ would
generally depend on $t$ if a standard quadrature rule were used to approximate
the indefinite integral.
Another attractive feature is that the Sinc indefinite integration can attain
an \emph{exponential} order of convergence, which is significantly faster
than polynomial-order convergence.
Using these features,
Muhammad et al.~\cite{muhammad05:_numer}
applied the Sinc indefinite integration to approximate $\Vol$ by $\Vol_N$
and defined a numerical solution $u_N$ satisfying
\[
 (\Ident - \Vol_N) u_N = g. 
\]
The procedure for obtaining the solution $u_N$
is called the Sinc-Nystr\"{o}m method,
which is described in Section~\ref{sec:nystroem}.
Theoretical analysis has established that the method achieves
a convergence rate of $\Order(\exp(-\sqrt{\pi d N}))$~\cite{okayama13:_theo},
where $d$ characterizes the size of the domain in which
the solution $u$ is analytic.

Another Sinc-based numerical solution for~\eqref{eq:Volterra-int}
was developed by
Rashidinia and Zarebnia~\cite{rashidinia07:_solut}.
They derived their method using a standard collocation procedure
based on the Sinc approximation, a function approximation formula
that also attains an exponential order of convergence.
Let $\mathcal{P}_N f$ denote the Sinc approximation of $f$.
As shown in this paper, the equation solved by their method
can be written symbolically as
\[
 (\Ident - \mathcal{P}_N \Vol_N) w_N = \mathcal{P}_N g .
\]
The procedure for obtaining the solution $w_N$
is called the Sinc-collocation method,
which is described in Section~\ref{sec:collocation}.
Although a theoretical error analysis of the method was
presented in~\cite{zarebnia10:_conv},
its convergence was not rigorously proved.

Yet another Sinc-based numerical solution for~\eqref{eq:Volterra-int}
was developed by Stenger~\cite{stenger93:_numer}.
Although his method was introduced more than a decade
before the above methods,
it has received considerably less attention.
One possible reason is that his method was originally formulated
for initial value problems
\begin{align*}
 u'(t) &= \tilde{k}(t)u(t) + \tilde{g}(t),\\
 u(a) &= u_a,
\end{align*}
which can be reduced to the following Volterra integral equation
of the second kind:
\begin{equation}
 u(t) - \int_a^t \tilde{k}(s) u(s)\diff s = g(t),
\label{eq:Volterra-initial-val}
\end{equation}
where $g(t)=u_a + \int_a^t\tilde{g}(s)\diff s$.
Because the kernel here ($\tilde{k}$) is a function of a single variable,
Stenger's original formulation does not appear to cover the general
two-variable kernel in~\eqref{eq:Volterra-int}.
Nevertheless, even without a theoretical justification,
his method can be implemented when the kernel is a function of two variables.
Its numerical solution, denoted by $v_N$, is determined in two steps:
(i) obtain the Sinc-Nystr\"{o}m solution $u_N$,
and (ii) apply the Sinc approximation to $u_N$.
Step (i) shows that Stenger's method is based on the
Sinc-Nystr\"{o}m method; however, $v_N$ is not equal to $u_N$ because of
step (ii).  The detailed procedure is described in
Section~\ref{sec:collocation}, where an explicit formulation for the general
case~\eqref{eq:Volterra-int} is presented for the first time.
The existing convergence result~\cite{stenger93:_numer} assumes that the
kernel is a function of a single variable.

Recent studies show that this diversity of discretization strategies
persists in Sinc-based solvers.
Sinc-Nystr\"{o}m methods have been developed for Fredholm integral equations
on infinite intervals~\cite{rahmoune20:_sinc}, whereas a direct
Sinc-collocation approach of the type developed by Rashidinia and Zarebnia
has been applied to Volterra integral equations with proportional
delay~\cite{mallick20:_numer}.
Stenger-type Sinc-collocation methods, in which an approximate function is
constructed from nodal values obtained through the corresponding
Sinc-Nystr\"{o}m system,
have also been studied for systems of initial value problems and Fredholm
integral equations~\cite{okayama18:_theo,okayama23:_theo}.
These studies illustrate that different discretizations and basis
constructions continue to be used depending on the target problem.
Although comparisons have been conducted for some problem classes, to the
best of the author's knowledge, the three formulations have not been
systematically compared for the general Volterra integral
equation~\eqref{eq:Volterra-int}.

This lack of a systematic comparison motivates the first objective of this
study.  From both theoretical and practical perspectives, we address the
following question: \emph{What are the differences and similarities among
these three methods, and which method is preferable?}
We first show that Stenger's method
can be regarded as another Sinc-collocation method.
We then prove that Stenger's method and the Rashidinia--Zarebnia method
coincide at the collocation points but are not generally equivalent.
Furthermore, we show that the two Sinc-collocation methods attain exactly
the same convergence rate,
$\Order(\sqrt{N}\exp(-\sqrt{\pi d \alpha N}))$,
where $0<\alpha\leq 1$ is a H\"{o}lder exponent.
From an implementation perspective, Stenger's method is preferable,
because its implementation is simpler than that of
the Rashidinia--Zarebnia method.

Thus, the remaining comparison is between the Sinc-Nystr\"{o}m method
and Stenger's Sinc-collocation method.
Although the theoretical convergence rate of Stenger's method is slightly
slower even when
$\alpha=1$, its approximate solution can be evaluated without repeated
evaluations of $k$, $g$, or the sine integral, and the numerical experiments
show lower computational cost for comparable accuracy.
We therefore select Stenger's formulation as the basis for the improvement
described next.

The second objective of this study is to improve Stenger's method.
All three methods described above employ the tanh transformation
\begin{equation}
t = \SEt(x) = \frac{b-a}{2}\tanh\left(\frac{x}{2}\right) + \frac{b+a}{2}
\label{eq:SEt}
\end{equation}
to map $\mathbb{R}$ onto the target interval $(a,b)$.
This transformation is required because Sinc numerical methods
are originally defined on the entire real axis $\mathbb{R}$.
For a finite interval, a variable transformation
such as~\eqref{eq:SEt} is therefore necessary.
We aim to improve Stenger's method by replacing the tanh
transformation with
\begin{equation}
t = \DEt(x)
 = \frac{b-a}{2}\tanh\left(\frac{\pi}{2}\sinh x\right) + \frac{b+a}{2},
\label{eq:DEt}
\end{equation}
which is called the double-exponential (DE) transformation.
The convergence rates of various Sinc numerical methods have been improved
by replacing the tanh transformation with
the DE transformation~\cite{mori01:_doubl,murota25:_doubl,sugihara04:_recen}.
Specifically, this replacement improves the convergence rate of the
Sinc-Nystr\"{o}m method to
$\Order(\log(2 d N)\exp(-\pi d N/\log(2 d N))/N)$~\cite{okayama13:_theo}.
Motivated by this result, we develop a new Sinc-collocation method
combined with the DE transformation.
We also analyze the proposed method theoretically and prove that
its convergence rate
is $\Order(\exp(-\pi d N/\log(2 d N/\alpha)))$,
which is significantly faster than that of Stenger's method.
Although this convergence rate is slightly slower than that of the
DE-Sinc-Nystr\"{o}m method, the proposed method avoids sine-integral
evaluation and exhibits a favorable balance between accuracy and
computational cost in the numerical experiments.
The proposed method is developed in two stages: first, Stenger's formulation
is selected through theoretical and computational comparison; second, its
tanh transformation is replaced by the DE transformation.

The main contributions are fourfold.
First, Stenger's method is reformulated and rigorously justified for general
two-variable kernels.
Second, its approximate function is shown to coincide with that of the
Rashidinia--Zarebnia method at the collocation points but not in general.
Third, a new RZ approximation operator connects the
Rashidinia--Zarebnia method to the Sinc-Nystr\"{o}m equation and enables a
rigorous convergence proof establishing the same root-exponential rate as
Stenger's method.
Fourth, the proposed DE-Sinc-collocation method attains almost exponential
convergence without requiring sine-integral evaluation.
The comparison also clarifies the convergence rates, evaluation costs, and
known conditioning properties of the methods, while identifying parameter
selection and the conditioning of the Rashidinia--Zarebnia matrix as issues
requiring further study.

The remainder of this paper is organized as follows.
Section~\ref{sec:preliminary} reviews the Sinc approximation and indefinite
integration, and Section~\ref{sec:nystroem} describes the corresponding
Sinc-Nystr\"{o}m methods.
Section~\ref{sec:collocation} compares the existing Sinc-collocation methods,
whereas Section~\ref{sec:de-collocation} develops the proposed DE method.
Section~\ref{sec:numer-result} presents the numerical experiments.
Sections~\ref{sec:proof-SE} and~\ref{sec:proof-DE} provide the proofs of the
new theoretical results, and Section~\ref{sec:conclusions} concludes the
paper.

\section{Preliminaries}
\label{sec:preliminary}

This section summarizes the Sinc approximation
and Sinc indefinite integration and their application
with the aid of the tanh or DE transformation.

\subsection{Sinc approximation and Sinc indefinite integration}

The Sinc numerical methods are generic names of numerical methods
based on the \emph{Sinc approximation},
expressed as
\begin{equation}
F(x) \approx \sum_{j=-N}^N F(jh)S(j,h)(x),\quad x\in\mathbb{R},
\label{eq:Sinc-approximation}
\end{equation}
where $h$ is a mesh size appropriately chosen depending on $N$,
and the basis function $S(j,h)$ is the so-called Sinc function
defined by
\[
 S(j,h)(x) =
\begin{cases}
 \dfrac{\sin(\pi(x - jh)/h)}{\pi(x - jh)/h} & (x \neq jh),\\
 1 & (x = jh).
\end{cases}
\]
Integrating both sides of~\eqref{eq:Sinc-approximation},
we obtain an approximation formula called
the \emph{Sinc indefinite integration} as
\begin{align}
\label{eq:Sinc-indefinite}
\int_{-\infty}^{\xi}F(x)\diff x
&\approx \sum_{j=-N}^N F(jh) \int_{-\infty}^{\xi}S(j,h)(x)\diff x
=\sum_{j=-N}^N F(jh) J(j,h)(\xi),\quad\xi\in\mathbb{R},
\end{align}
where $J(j,h)$ is defined by
\[
 J(j,h)(x) = h \left\{
\frac{1}{2} + \frac{1}{\pi}\Si\left[\frac{\pi(x - jh)}{h}\right]
\right\},
\]
where $\Si(x)$ is the sine integral defined by
$\Si(x)=\int_0^x \{(\sin t) / t\}\diff t$.

\subsection{SE-Sinc approximation and SE-Sinc indefinite integration}

To use the approximation formulas~\eqref{eq:Sinc-approximation}
and~\eqref{eq:Sinc-indefinite},
the function $F(x)$ must be defined over the entire real line $\mathbb{R}$.
When the function $f(t)$ is
defined over the finite interval $(a, b)$,
a variable transformation is required to map $\mathbb{R}$
onto $(a, b)$.
For the purpose,
the tanh transformation $t=\SEt(x)$ defined in~\eqref{eq:SEt}
is widely employed.
The change of variable ($t=\SEt(x)$) enables us
to apply~\eqref{eq:Sinc-approximation} by setting $F(x)=f(\SEt(x))$.
Introducing $\tSE_j = \SEt(jh)$ and $\SEtInv(t)=\{\SEt\}^{-1}(t)$,
we express the obtained formula as
\begin{align}
 f(t)
& \approx \sum_{j=-N}^N f(\tSE_j)S(j,h)(\SEtInv(t)),\quad t\in(a, b).
\label{eq:SE-Sinc-approximation}
\end{align}
This approximation is referred to as the SE-Sinc approximation
in this paper.
Similarly, applying $s=\SEt(x)$ and
setting $F(x)=f(\SEt(x))$ in~\eqref{eq:Sinc-indefinite},
we obtain
\begin{align}
 \int_a^t f(s) \diff s
= \int_{-\infty}^{\SEtInv(t)} f(\SEt(x))\SEtDiv(x)\diff x
\approx \sum_{j=-N}^N f(\tSE_j)\SEtDiv(jh) J(j,h)(\SEtInv(t)),
\quad t\in(a, b),
\label{eq:SE-Sinc-indefinite}
\end{align}
which is referred to as the SE-Sinc indefinite integration
in this paper.
If $F(x)=f(\psi(x))$ is analytic on the strip complex domain
\[
 \domD_d = \left\{\zeta\in\mathbb{C} : |\Im\zeta| < d\right\}
\]
for a positive constant $d$,
then both approximations performs highly accurately.
In other words, $f(t)$ should be analytic on the transformed domain
\[
 \SEt(\domD_d) = \left\{z=\SEt(\zeta) : \zeta\in\domD_d\right\},
\]
which is a simply-connected domain.
Convergence theorems of the two approximations were provided as follows.

\begin{theorem}[Stenger~{\cite[Theorem~4.2.5]{stenger93:_numer}}]
\label{thm:SE-Sinc-approx}
Assume that $f$ is analytic on $\SEt(\domD_d)$
for $d$ with $0<d<\pi$, and
there exists constants $K$ and $\alpha$ such that
\begin{equation}
 |f(z)|\leq K |z - a|^{\alpha}|b - z|^{\alpha}
\label{eq:LC}
\end{equation}
holds for all $z\in\SEt(\domD_d)$. Let $N$ be a positive integer,
and let $h$ be selected by the formula
\begin{equation}
\label{eq:h-SE}
 h = \sqrt{\frac{\pi d}{\alpha N}}.
\end{equation}
Then, there exists a constant $C$ independent of $N$ such that
\[
 \max_{t\in [a, b]}
\left|f(t) - \sum_{j=-N}^N f(\tSE_j)S(j,h)(\SEtInv(t))\right|
\leq C \sqrt{N} \rme^{-\sqrt{\pi d \alpha N}}.
\]
\end{theorem}

\begin{theorem}[Okayama et al.~{\cite[Theorem~2.9]{okayama09:_error}}]
\label{thm:SE-Sinc-indefinite}
Assume that $f$ is analytic on $\SEt(\domD_d)$
for $d$ with $0<d<\pi$, and
there exists constants $K$ and $\alpha$ such that
\begin{equation}
\label{eq:QC}
 |f(z)|\leq K |z - a|^{\alpha - 1}|b - z|^{\alpha - 1}
\end{equation}
holds for all $z\in\SEt(\domD_d)$. Let $N$ be a positive integer,
and let $h$ be selected by the formula~\eqref{eq:h-SE}.
Then, there exists a constant $C$ independent of $N$ such that
\[
 \max_{t\in [a, b]}
\left|
\int_a^t f(s)\diff s - \sum_{j=-N}^N f(\tSE_j)\SEtDiv(jh)J(j,h)(\SEtInv(t))
\right|
\leq C \rme^{-\sqrt{\pi d \alpha N}}.
\]
\end{theorem}

\subsection{DE-Sinc approximation and DE-Sinc indefinite integration}

The SE-Sinc approximation~\eqref{eq:SE-Sinc-approximation}
and SE-Sinc indefinite integration~\eqref{eq:SE-Sinc-indefinite}
employ the tanh transformation~\eqref{eq:SEt} to
map $\mathbb{R}$ onto the finite interval $(a, b)$.
The DE transformation~\eqref{eq:DEt} also plays the same role,
and allows for the replacement of $\SEt$ with $\DEt$ in both formulas.
On the basis of this idea,
introducing $\tDE_j = \DEt(jh)$ and $\DEtInv(t)=\{\DEt\}^{-1}(t)$,
we can derive the following formulas
\begin{align}
\label{eq:DE-Sinc-approximation}
 f(t)
& \approx \sum_{j=-N}^N f(\tDE_j)S(j,h)(\DEtInv(t)),\quad t\in(a, b),\\
\label{eq:DE-Sinc-indefinite}
 \int_a^t f(s) \diff s
&\approx \sum_{j=-N}^N f(\tDE_j)\DEtDiv(jh) J(j,h)(\DEtInv(t)),
\quad t\in(a, b),
\end{align}
which are referred to as the DE-Sinc approximation
and DE-Sinc indefinite integration, respectively.
For the formulas~\eqref{eq:DE-Sinc-approximation}
and~\eqref{eq:DE-Sinc-indefinite},
$f(t)$ should be analytic on the transformed domain
\[
 \DEt(\domD_d) = \left\{z=\DEt(\zeta) : \zeta\in\domD_d\right\},
\]
which forms a Riemann surface.
Convergence theorems of the two approximations were provided as follows.

\begin{theorem}[Tanaka et al.~{\cite[Theorem~3.1]{tanaka09:_desinc}}]
Assume that $f$ is analytic on $\DEt(\domD_d)$
for $d$ with $0<d<\pi/2$, and
there exists constants $K$ and $\alpha$ such that~\eqref{eq:LC}
holds for all $z\in\DEt(\domD_d)$. Let $N$ be a positive integer,
and let $h$ be selected by the formula
\begin{equation}
\label{eq:h-DE}
 h = \frac{\log(2 d N/\alpha)}{N}.
\end{equation}
Then, there exists a constant $C$ independent of $N$ such that
\[
 \max_{t\in [a, b]}
\left|f(t) - \sum_{j=-N}^N f(\tDE_j)S(j,h)(\DEtInv(t))\right|
\leq C \rme^{-\pi d N/\log(2 d N/\alpha)}.
\]
\end{theorem}

\begin{theorem}[Okayama et al.~{\cite[Theorem~2.16]{okayama09:_error}}]
\label{thm:DE-Sinc-indefinite}
Assume that $f$ is analytic on $\DEt(\domD_d)$
for $d$ with $0<d<\pi/2$, and
there exists constants $K$ and $\alpha$ such that~\eqref{eq:QC}
holds for all $z\in\DEt(\domD_d)$. Let $N$ be a positive integer,
and let $h$ be selected by the formula~\eqref{eq:h-DE}.
Then, there exists a constant $C$ independent of $N$ such that
\begin{align*}
 \max_{t\in [a, b]}
\left|
\int_a^t f(s)\diff s - \sum_{j=-N}^N f(\tDE_j)\DEtDiv(jh)J(j,h)(\DEtInv(t))
\right|
\leq C \frac{\log(2 d N/\alpha)}{N}\rme^{-\pi d N/\log(2 d N/\alpha)}.
\end{align*}
\end{theorem}

\subsection{Generalized SE/DE-Sinc approximation}

In the convergence theorems of the SE/DE-Sinc approximation,
the condition~\eqref{eq:LC} is assumed.
This condition requires $f(t)$ to be zero at the endpoints $t=a$ and $t=b$,
which seems an impractical condition.
To address this issue,
using auxiliary functions
\[
 \omega_a(t) = \frac{b - t}{b - a},\quad
 \omega_b(t) = \frac{t - a}{b - a},
\]
and setting
$\tilde{f}^{\textSE}_N(t)=f(t) -f(\tSE_{-N})\omega_a(t) -f(\tSE_N)\omega_b(t)$,
Stenger~\cite{stenger93:_numer,stenger00:_summar}
proposed to apply the SE-Sinc approximation to $\tilde{f}^{\textSE}_N$.
If we define an approximation operator $\ProjSE: C([a,b])\to C([a, b])$ as
\begin{align}
\label{eq:ProjSE}
(\ProjSE f)(t)
=
f(\tSE_{-N})\omega_a(t) + f(\tSE_N)\omega_b(t)
 + \sum_{j=-N}^N \tilde{f}^{\textSE}_N(\tSE_j) S(j,h)(\SEtInv(t)),
\end{align}
then the approximation is expressed as $f\approx \ProjSE f$.
This approximation is referred to as the generalized SE-Sinc approximation
in this paper.
Notably, $\ProjSE$ satisfies
the interpolation property, that is,
$f(\tSE_i) = (\ProjSE f)(\tSE_i)$ ($i=-N,\,\ldots,\,N$).

Similarly, setting
$\tilde{f}^{\textDE}_N(t)=f(t) -f(\tDE_{-N})\omega_a(t) -f(\tDE_N)\omega_b(t)$,
we may apply the DE-Sinc approximation to $\tilde{f}^{\textDE}_N$.
If we define an approximation operator $\ProjDE: C([a,b])\to C([a, b])$ as
\begin{align}
\label{eq:ProjDE}
(\ProjDE f)(t)
=
f(\tDE_{-N})\omega_a(t) + f(\tDE_N)\omega_b(t)
 + \sum_{j=-N}^N \tilde{f}^{\textDE}_N(\tDE_j) S(j,h)(\DEtInv(t)),
\end{align}
then the approximation is expressed as $f\approx \ProjDE f$.
This approximation is referred to as the generalized DE-Sinc approximation
in this paper.
$\ProjDE$ also satisfies
the interpolation property, that is,
$f(\tDE_i) = (\ProjDE f)(\tDE_i)$ ($i=-N,\,\ldots,\,N$).

The convergence theorems of the two approximations are described
using the following function spaces.

\begin{definition}
Let $\domD$ be a bounded and simply-connected domain
(or Riemann surface).
Then, $\Hinf(\domD)$ denotes the family of functions $f$
analytic on $\domD$ such that the norm $\|f\|_{\Hinf(\domD)}$ is finite,
where
\[
 \|f\|_{\Hinf(\domD)} = \sup_{z\in\domD}|f(z)|.
\]
\end{definition}
\begin{definition}
Let $\alpha$ be a positive constant, and
let $\domD$ be a bounded and simply-connected domain
(or Riemann surface) that satisfies $(a, b)\subset \domD$.
Then, $\MC_{\alpha}(\domD)$ denotes the family of functions $f\in\Hinf(\domD)$
for which there exists a constant $L$ such that for all $z\in\domD$,
\begin{align*}
 |f(z) - f(a)| &\leq L |z - a|^{\alpha},\\
 |f(b) - f(z)| &\leq L |b - z|^{\alpha}.
\end{align*}
\end{definition}

For a two-variable function $k$, we say that $k(\cdot,w)$ belongs to
$\MC_{\alpha}(\domD)$ uniformly with respect to $w\in\domD$ if there exists
a constant $L_k$ independent of $w$ such that, for all $z,w\in\domD$,
\begin{align*}
 |k(z,w) - k(a,w)| &\leq L_k |z-a|^{\alpha},\\
 |k(b,w) - k(z,w)| &\leq L_k |b-z|^{\alpha}.
\end{align*}

This function space $\MC_{\alpha}(\domD)$ only requires
the H\"{o}lder continuity at the endpoints instead of
the zero-boundary condition by~\eqref{eq:LC}.
Convergence theorems of the two approximations were provided as follows.
Here, $\|\cdot\|_{C([a,b])}$ denotes the usual uniform norm over $[a, b]$.

\begin{theorem}[Okayama~{\cite[Theorem~3]{okayama13:_note}}]
\label{thm:SE-Sinc-general}
Assume that $f\in\MC_{\alpha}(\SEt(\domD_d))$ for $d$ with $0<d<\pi$.
Let $N$ be a positive integer,
and let $h$ be selected by the formula~\eqref{eq:h-SE}.
Then, there exists a constant $C$ independent of $N$ such that
\[
 \|f - \ProjSE f\|_{C([a,b])}
\leq C \sqrt{N}\rme^{-\sqrt{\pi d \alpha N}}.
\]
\end{theorem}
\begin{theorem}[Okayama~{\cite[Theorem~6]{okayama13:_note}}]
\label{thm:DE-Sinc-general}
Assume that $f\in\MC_{\alpha}(\DEt(\domD_d))$ for $d$ with $0<d<\pi/2$.
Let $N$ be a positive integer,
and let $h$ be selected by the formula~\eqref{eq:h-DE}.
Then, there exists a constant $C$ independent of $N$ such that
\[
 \|f - \ProjDE f\|_{C([a,b])}
\leq C \rme^{-\pi d N/\log(2 d N/\alpha)}.
\]
\end{theorem}

\section{Sinc-Nystr\"{o}m methods}
\label{sec:nystroem}

This section describes the Sinc-Nystr\"{o}m methods
developed by Muhammad et al.~\cite{muhammad05:_numer}.
The first method employs the tanh transformation~\eqref{eq:SEt}
as a variable transformation,
while the second method employs the DE transformation~\eqref{eq:DEt}.

\subsection{SE-Sinc-Nystr\"{o}m method}

By applying the SE-Sinc indefinite integration~\eqref{eq:SE-Sinc-indefinite}
to the integral in the given equation~\eqref{eq:Volterra-int},
we obtain an approximated equation as
\begin{equation}
\label{eq:SE-Sinc-Nystroem}
 \uSEn(t)
= g(t)
+\sum_{j=-N}^N k(t,\tSE_j)\uSEn(\tSE_j)\SEtDiv(jh) J(j,h)(\SEtInv(t)).
\end{equation}
The approximated solution $\uSEn$ is determined
once the unknown coefficients $\uSEn(\tSE_j)$
on the right-hand side are obtained. To this end,
$2N+1$ sampling points are set at
$t=\tSE_i$ $(i=-N,\,-N+1,\,\ldots,\,N)$ in~\eqref{eq:SE-Sinc-Nystroem}
as
\begin{equation}
\label{eq:SE-Sinc-Nystroem-system}
 \uSEn(\tSE_i)
= g(\tSE_i)
+\sum_{j=-N}^N k(\tSE_i,\tSE_j)\uSEn(\tSE_j)\SEtDiv(jh)
 J(j,h)(ih),
\quad i=-N,\,\ldots,\,N,
\end{equation}
which is a system of linear equations.
This system is expressed in a matrix-vector form as follows.
Let $n=2N+1$,
let $I_n$ be an identity matrix of order $n$,
and let $\kSEn$ be an $n\times n$ matrix
whose $(i, j)$-th element is
\[
 \left(\kSEn\right)_{ij}
= k(\tSE_i,\tSE_j)\SEtDiv(jh)h \delta_{i-j}^{(-1)},
\quad i= -N,\,\ldots,\,N,\quad j=-N,\,\ldots,\,N,
\]
where $\delta_k^{(-1)} = (1/2) + \sigma_k$, where $\sigma_k$ is defined by
\[
\sigma_k = \int_0^k\frac{\sin(\pi x)}{\pi x}\diff x
=\frac{1}{\pi}\Si(\pi k).
\]
Furthermore, let $\gSEn$ and $\mathbd{u}_n^{\textSE}$
be $n$-dimensional vectors defined by
\begin{align*}
 \gSEn
 &= [g(\tSE_{-N}),\,g(\tSE_{-N+1}),\,\ldots,\,g(\tSE_N)]^{\mathrm{T}},\\
 \mathbd{u}_n^{\textSE}
 &= [\uSEn(\tSE_{-N}),\,\uSEn(\tSE_{-N+1}),\,\ldots,\,\uSEn(\tSE_N)]^{\mathrm{T}}.
\end{align*}
Then, the system~\eqref{eq:SE-Sinc-Nystroem-system} is expressed as
\begin{equation}
\label{eq:SE-Sinc-Nystroem-linear-eq}
 (I_n - \kSEn) \mathbd{u}_n^{\textSE} = \gSEn.
\end{equation}
By solving~\eqref{eq:SE-Sinc-Nystroem-linear-eq},
we obtain the unknown coefficients $\mathbd{u}_n^{\textSE}$,
from which the approximated solution $\uSEn$ is determined
by~\eqref{eq:SE-Sinc-Nystroem}.
This procedure is called the SE-Sinc-Nystr\"{o}m method.
Its convergence theorem was provided as follows.

\begin{theorem}[Okayama et al.~{\cite[Theorem~3.4]{okayama13:_theo}}]
\label{thm:SE-Sinc-Nystroem}
Let $d$ be a positive constant with $d<\pi$.
Assume that $g$, $k(z,\cdot)$ and $k(\cdot,w)$
belong to $\Hinf(\SEt(\domD_d))$ for all $z$, $w\in\SEt(\domD_d)$.
Furthermore, assume that
$g$, $k(t,\cdot)$ and $k(\cdot,s)$ belong to $C([a,b])$
for all $t$, $s\in [a, b]$.
Let $h$ be selected by the formula~\eqref{eq:h-SE} with $\alpha=1$.
Then, there exists a positive integer $N_0$
such that for all $N\geq N_0$,
the coefficient matrix $(I_n - \kSEn)$ is invertible.
Furthermore, there exists a constant $C$ independent of $N$
such that for all $N\geq N_0$,
\[
 \|u - \uSEn\|_{C([a,b])} \leq C \rme^{-\sqrt{\pi d N}}.
\]
\end{theorem}

\subsection{DE-Sinc-Nystr\"{o}m method}

Muhammad et al.~\cite{muhammad05:_numer} also considered
another method by replacing $\SEt$ with $\DEt$
in the SE-Sinc-Nystr\"{o}m method.
Applying the DE-Sinc indefinite integration~\eqref{eq:DE-Sinc-indefinite}
to the integral in the given equation~\eqref{eq:Volterra-int},
we obtain an approximated equation as
\begin{equation}
\label{eq:DE-Sinc-Nystroem}
 \uDEn(t)
= g(t)
+\sum_{j=-N}^N k(t,\tDE_j)\uDEn(\tDE_j)\DEtDiv(jh) J(j,h)(\DEtInv(t)).
\end{equation}
The approximated solution $\uDEn$ is determined
once the unknown coefficients $\uDEn(\tDE_j)$
on the right-hand side are obtained. To this end,
$2N+1$ sampling points are set at
$t=\tDE_i$ $(i=-N,\,-N+1,\,\ldots,\,N)$ in~\eqref{eq:DE-Sinc-Nystroem}.
This leads to a system of linear equations
\begin{equation}
\label{eq:DE-Sinc-Nystroem-linear-eq}
 (I_n - \kDEn) \mathbd{u}_n^{\textDE} = \gDEn,
\end{equation}
where $n=2N+1$, and $\kDEn$ is an $n\times n$ matrix
whose $(i, j)$-th element is
\[
 \left(\kDEn\right)_{ij}
= k(\tDE_i,\tDE_j)\DEtDiv(jh)h \delta_{i-j}^{(-1)},
\quad i= -N,\,\ldots,\,N,\quad j=-N,\,\ldots,\,N,
\]
and $\gDEn$ and $\mathbd{u}_n^{\textDE}$
are $n$-dimensional vectors defined by
\begin{align*}
 \gDEn
 &= [g(\tDE_{-N}),\,g(\tDE_{-N+1}),\,\ldots,\,g(\tDE_N)]^{\mathrm{T}},\\
 \mathbd{u}_n^{\textDE}
 &= [\uDEn(\tDE_{-N}),\,\uDEn(\tDE_{-N+1}),\,\ldots,\,\uDEn(\tDE_N)]^{\mathrm{T}}.
\end{align*}
By solving~\eqref{eq:DE-Sinc-Nystroem-linear-eq},
we obtain the unknown coefficients $\mathbd{u}_n^{\textDE}$,
from which the approximated solution $\uDEn$ is determined
by~\eqref{eq:DE-Sinc-Nystroem}.
This procedure is called the DE-Sinc-Nystr\"{o}m method.
Its convergence theorem was provided as follows.

\begin{theorem}[Okayama et al.~{\cite[Theorem~3.5]{okayama13:_theo}}]
\label{thm:DE-Sinc-Nystroem}
Let $d$ be a positive constant with $d<\pi/2$.
Assume that $g$, $k(z,\cdot)$ and $k(\cdot,w)$
belong to $\Hinf(\DEt(\domD_d))$ for all $z$, $w\in\DEt(\domD_d)$.
Furthermore, assume that
$g$, $k(t,\cdot)$ and $k(\cdot,s)$ belong to $C([a,b])$
for all $t$, $s\in [a, b]$.
Let $h$ be selected by the formula~\eqref{eq:h-DE} with $\alpha=1$.
Then, there exists a positive integer $N_0$
such that for all $N\geq N_0$,
the coefficient matrix $(I_n - \kDEn)$ is invertible.
Furthermore, there exists a constant $C$ independent of $N$
such that for all $N\geq N_0$,
\[
 \|u - \uDEn\|_{C([a,b])}
 \leq C \frac{\log(2 d N)}{N}\rme^{-\pi d N/\log(2 d N)}.
\]
\end{theorem}

\section{Existing Sinc-collocation methods}
\label{sec:collocation}

This section describes two different Sinc-collocation methods
developed by
Stenger~\cite{stenger93:_numer}
and Rashidinia and Zarebnia~\cite{rashidinia07:_solut}.
Both methods employ the tanh transformation~\eqref{eq:SEt}
as a variable transformation,
but their procedures are not identical.

\subsection{Sinc-collocation method by Stenger}

As explained in Section~\ref{sec:introduction},
Stenger derived his method for~\eqref{eq:Volterra-initial-val},
where the kernel $\tilde{k}$ is a function of a single variable.
However, his method can be easily derived for~\eqref{eq:Volterra-int}
as follows.
His method is closely related to the SE-Sinc-Nystr\"{o}m method,
which is described in the previous section.
First, solve the linear system~\eqref{eq:SE-Sinc-Nystroem-linear-eq}
and obtain~$\mathbd{u}_n^{\textSE}$.
Then, his approximated solution $\vSEn$ is expressed as
the generalized SE-Sinc approximation of $\uSEn$, i.e.,
\begin{align}
 \vSEn(t)
 &= \ProjSE[\uSEn](t)\nonumber\\
 &= \uSEn(\tSE_{-N})\omega_a(t) + \uSEn(\tSE_{N})\omega_b(t) +
\sum_{j=-N}^N
\left\{\uSEn(\tSE_j) - \uSEn(\tSE_{-N})\omega_a(\tSE_j)
 - \uSEn(\tSE_N)\omega_b(\tSE_j)\right\}S(j,h)(\SEtInv(t)) ,
\label{eq:SE-Sinc-collocation}
\end{align}
where $\ProjSE$ is defined by~\eqref{eq:ProjSE}.

The solution $\vSEn$ is also obtained
by the standard collocation procedure as follows.
Set the approximate solution $\vSEn$ as~\eqref{eq:SE-Sinc-collocation},
where $\uSEn(\tSE_j)$ $(j=-N,\,\ldots,\,N)$ are
regarded as unknown coefficients.
Substitute $\vSEn$ into the given equation~\eqref{eq:Volterra-int},
with approximating the Volterra integral operator $\Vol$ by $\VolSEn$, where
\begin{equation}
\label{eq:VolSEn}
 \VolSEn [f](t)
=\sum_{j=-N}^N k(t,\tSE_j)f(\tSE_j)\SEtDiv(jh)J(j,h)(\SEtInv(t)),
\end{equation}
which is the SE-Sinc indefinite integration of $\Vol f$.
Then, setting $n=2N+1$ sampling points at $t=\tSE_i$
$(i=-N,\,-N+1,\,\ldots,\,N)$, we obtain
the same system of linear equations as~\eqref{eq:SE-Sinc-Nystroem-linear-eq}.
Thus, Stenger's method can be regarded as a collocation method
utilizing the generalized SE-Sinc approximation,
namely, SE-Sinc-collocation method.

\subsection{Sinc-collocation method by Rashidinia and Zarebnia}

Rashidinia and Zarebnia derived their method
by the standard collocation procedure,
but they considered their approximated solution $\vRZn$
in different manners in the following four cases.
\begin{enumerate}
 \item[(I)] If $u(a)=u(b)=0$, set $\vRZn$ as
\[
 \vRZn(t) = \sum_{j=-N}^{N} c_{j} S(j,h)(\SEtInv(t)).
\]
 \item[(II)] If $u(a)\neq 0$ and $u(b)=0$, set $\vRZn$ as
\[
 \vRZn(t) = c_{-N}\omega_a(t) + \sum_{j=-N+1}^{N} c_{j} S(j,h)(\SEtInv(t)).
\]
 \item[(III)] If $u(a)= 0$ and $u(b)\neq 0$, set $\vRZn$ as
\[
 \vRZn(t) = \sum_{j=-N}^{N-1} c_{j} S(j,h)(\SEtInv(t)) + c_{N}\omega_b(t).
\]
 \item[(IV)] If $u(a)\neq 0$ and $u(b)\neq 0$, set $\vRZn$ as
\end{enumerate}
\begin{equation}
\label{eq:vRZn}
 \vRZn(t) = c_{-N}\omega_a(t)
 + \sum_{j=-N+1}^{N-1} c_{j} S(j,h)(\SEtInv(t)) + c_{N}\omega_b(t).
\end{equation}
To obtain the unknown coefficients
$\mathbd{c}_{n} = [c_{-N},\,c_{-N+1},\,\ldots,\,c_{N}]^{\mathrm{T}}$,
where $n=2N+1$,
they substituted $\vRZn$ into the given equation~\eqref{eq:Volterra-int},
with approximating the Volterra integral operator $\Vol$ by $\VolSEn$.
Then, setting $n$ sampling points at $t=\tSE_i$
$(i=-N,\,-N+1,\,\ldots,\,N)$,
they derived a system of linear equations
in each of the four cases: (I)--(IV).
For example, in the case (I), the resulting system is expressed as
\[
 (I_n - V_{n}^{\textSE})\mathbd{c}_n = \mathbd{g}_n^{\textSE}.
\]
Particularly, in the case (IV), the resulting system is expressed as
\begin{equation}
\label{eq:RZ-linear-eq}
 (E_n^{\textRZ} - V_{n}^{\textRZ})\mathbd{c}_n = \mathbd{g}_n^{\textSE},
\end{equation}
where $E_n^{\textRZ}$ and $V_n^{\textRZ}$ are
$n\times n$ matrices defined by
\begin{align*}
 E_n^{\textRZ}
&= \left[
   \begin{array}{@{\,}c|ccc|c@{\,}}
   \omega_a(\tSE_{-N})        & 0    & \cdots &0  & \omega_b(\tSE_{-N}) \\
   \omega_a(\tSE_{-N+1}) & 1      &        &0 & \omega_b(\tSE_{-N+1})\\
   \vdots     &        & \ddots &       & \vdots \\
   \omega_a(\tSE_{N-1}) & 0 &        &1      & \omega_b(\tSE_{N-1}) \\
   \omega_a(\tSE_{N}) & 0      & \cdots &0      & \omega_b(\tSE_{N})
   \end{array}
   \right], \\
  V_n^{\textRZ}
&= \left[
   \begin{array}{@{\,}l|clc|l@{\,}}
   &\cdots
   & k(\tSE_{-N},\tSE_j)\SEtDiv(jh) h \delta_{-N-j}^{(-1)}
   &\cdots
   & \\
   &\cdots
   & k(\tSE_{-N+1},\tSE_j)\SEtDiv(jh) h \delta_{-N+1-j}^{(-1)}
   &\cdots
   & \\
    \multicolumn{1}{c|}{\mathbd{p}_n^{\textRZ}} & & \multicolumn{1}{c}{\vdots}
   & & \multicolumn{1}{c}{\mathbd{q}_n^{\textRZ}}\\
   &\cdots
   & k(\tSE_{N-1},\tSE_j)\SEtDiv(jh) h \delta_{N-1-j}^{(-1)}
   &\cdots
   & \\
   &\cdots
   & k(\tSE_{N},\tSE_j)\SEtDiv(jh) h \delta_{N-j}^{(-1)}
   &\cdots
   &
   \end{array}
   \right],
\end{align*}
where $\mathbd{p}_n^{\textRZ}$
and $\mathbd{q}_n^{\textRZ}$
are $n$-dimensional vectors defined by
\begin{align*}
\mathbd{p}_n^{\textRZ}
&= [ \VolSEn[\omega_a](\tSE_{-N}),\,
\VolSEn[\omega_a](\tSE_{-N+1}),\,\ldots,\,
\VolSEn[\omega_a](\tSE_{N})]^{\mathrm{T}}, \\
\mathbd{q}_n^{\textRZ}
&= [ \VolSEn[\omega_b](\tSE_{-N}),\,
\VolSEn[\omega_b](\tSE_{-N+1}),\,\ldots,\,
\VolSEn[\omega_b](\tSE_{N})]^{\mathrm{T}}.
\end{align*}
This is the SE-Sinc-collocation method by Rashidinia and Zarebnia.
In the case (I), the following error analysis was provided.

\begin{theorem}[Zarebnia and Rashidinia~{\cite[Theorem~3]{zarebnia10:_conv}}]
\label{thm:Zarebnia-Rashidinia}
Let $\alpha$ and $d$ be positive constants with $d<\pi$.
Assume that the solution $u$ in~\eqref{eq:Volterra-int}
satisfies the assumptions of Theorem~\ref{thm:SE-Sinc-approx}.
Furthermore, assume that $k(t,\cdot)$
satisfies the assumptions of Theorem~\ref{thm:SE-Sinc-indefinite}
for all $t\in [a, b]$.
Then, there exists a constant $C$ independent of $N$ such that
\[
 \|u - \vRZn\|_{C([a,b])}
\leq C \|(I_n- V_{n}^{\textSE})^{-1}\|_2\sqrt{N}\rme^{-\sqrt{\pi d \alpha N}}.
\]
\end{theorem}

However,
this theorem does not prove the convergence of $\vRZn$,
because there exists an unestimated
term $\|(I_n - V_{n}^{\textSE})^{-1}\|_2$, which clearly depends on $N$.
For the cases (II)--(IV), no error analysis has been provided thus far.

Moreover,
in a practical situation,
it is hard to determine whether $u$ is zero or not at the endpoints.
This is because the solution $u$ is an unknown function to be determined.
The idea to address the issue was presented
for Fredholm integral equations~\cite{okayama1x:_improv};
set the approximate solution $\vRZn$ as~\eqref{eq:vRZn} in any cases.
In other words, we may treat the case (IV) as a general case.
This idea can be employed
for Volterra integral equations~\eqref{eq:Volterra-int}.
Therefore, as a method by Rashidinia and Zarebnia,
this study adopts the following procedure:
(i) solve the linear system~\eqref{eq:RZ-linear-eq},
and (ii) obtain the approximate solution by~\eqref{eq:vRZn}.

\subsection{Main result 1: Relation between the two methods and their convergence}

Any relation between Stenger's method $(\vSEn)$
and Rashidinia--Zarebnia's method $(\vRZn)$ has not been investigated thus far.
Furthermore, convergence of the two methods has not been rigorously proved.
As a first contribution of this paper,
we show the relation between the two methods as follows.
The proof is provided in Section~\ref{sec:proof-equivalence}.

\begin{theorem}
\label{thm:equivalence}
Let $\vSEn$ be a function defined by~\eqref{eq:SE-Sinc-collocation},
where $\mathbd{u}_n^{\textSE}$ is determined
by solving the linear system~\eqref{eq:SE-Sinc-Nystroem-linear-eq}.
Furthermore,
let $\vRZn$ be a function defined by~\eqref{eq:vRZn},
where $\mathbd{c}_{n}$ is determined
by solving the linear system~\eqref{eq:RZ-linear-eq}.
Then, it holds that
\[
 \vSEn(\tSE_i) = \vRZn(\tSE_i),\quad i=-N,\,-N+1,\,\ldots,\,N,
\]
but generally $\vSEn\neq \vRZn$.
\end{theorem}

Subsequently, we provide the convergence theorems of
the two methods as follows.
Their proofs are provided in Sections~\ref{sec:proof-SE-Sinc}
and~\ref{sec:proof-RZ-Sinc}.

\begin{theorem}
\label{thm:SE-Sinc-collocation}
Let $\alpha$ and $d$ be positive constants with
$\alpha\leq 1$ and $d<\pi$.
Assume that
the assumptions on $g$ and $k$ in Theorem~\ref{thm:SE-Sinc-Nystroem}
are satisfied.
Furthermore, assume that $g\in\MC_{\alpha}(\SEt(\domD_d))$ and that
$k(\cdot,w)$ belongs to $\MC_{\alpha}(\SEt(\domD_d))$ uniformly with
respect to $w\in\SEt(\domD_d)$.
Let $h$ be selected by the formula~\eqref{eq:h-SE}.
Then, there exists a positive integer $N_0$
such that for all $N\geq N_0$,
the coefficient matrix $(I_n - \kSEn)$ is invertible.
Furthermore, there exists a constant $C$ independent of $N$
such that for all $N\geq N_0$,
\[
 \|u - \vSEn\|_{C([a,b])} \leq C \sqrt{N} \rme^{-\sqrt{\pi d \alpha N}}.
\]
\end{theorem}

\begin{theorem}
\label{thm:RZ-Sinc-collocation}
Assume that the assumptions of Theorem~\ref{thm:SE-Sinc-collocation}
are satisfied.
Then, there exists a positive integer $N_0$
such that for all $N\geq N_0$,
the coefficient matrix $(E_n^{\textRZ} - V_{n}^{\textRZ})$ is invertible.
Furthermore, there exists a constant $C$ independent of $N$
such that for all $N\geq N_0$,
\[
 \|u - \vRZn\|_{C([a,b])} \leq C \sqrt{N} \rme^{-\sqrt{\pi d \alpha N}}.
\]
\end{theorem}

\begin{remark}
In view of Theorems~\ref{thm:Zarebnia-Rashidinia}
and~\ref{thm:RZ-Sinc-collocation},
one might assume that Theorem~\ref{thm:RZ-Sinc-collocation}
is proved by bounding $\|(I_n - \kSEn)^{-1}\|_2$ uniformly for $N$.
However, this is not the case; see Section~\ref{sec:proof-SE} for details.
\end{remark}

Theorems~\ref{thm:SE-Sinc-collocation} and~\ref{thm:RZ-Sinc-collocation}
reveal that both methods achieve the same convergence rate.
Therefore, users may prefer Stenger's method
because the implementation of the method by Rashidinia and Zarebnia
is rather complicated.
This complication also causes difficulty in extension to the \emph{system} of
Volterra integral equations.

Regarding numerical stability,
the condition number of the coefficient matrix $(I_n-\kSEn)$
in the infinity norm is known to be uniformly bounded with respect to $N$
under the assumptions for the SE-Sinc-Nystr\"{o}m
method~\cite[Theorem~3.7]{okayama13:_theo}.
Because Stenger's SE-Sinc-collocation method uses exactly the same
coefficient matrix, it inherits this uniform well-conditioning.
By contrast, the Rashidinia--Zarebnia method uses a different coefficient
matrix $(E_n^{\textRZ}-V_n^{\textRZ})$; therefore, the corresponding
boundedness does not follow directly from the existing result.

For this reason, in the next section,
we consider the improvement of Stenger's method.

\section{Sinc-collocation method combined with the DE transformation}
\label{sec:de-collocation}

The SE-Sinc-collocation method described in the previous section
employs the tanh transformation as a variable transformation.
In this section, a new method is derived
by replacing the tanh transformation with the DE transformation.
Then, its convergence theorem is stated.

\subsection{Derivation of the DE-Sinc-collocation method}

First, solve the linear system~\eqref{eq:DE-Sinc-Nystroem-linear-eq}
and obtain~$\mathbd{u}_n^{\textDE}$.
Then, the approximated solution $\vDEn$ is expressed as
the generalized DE-Sinc approximation of $\uDEn$, i.e.,
\begin{align}
 \vDEn(t)
 &= \ProjDE[\uDEn](t)\nonumber\\
& = \uDEn(\tDE_{-N})\omega_a(t) + \uDEn(\tDE_{N})\omega_b(t)
+ \sum_{j=-N}^N
\left\{\uDEn(\tDE_j) - \uDEn(\tDE_{-N})\omega_a(\tDE_j)
 - \uDEn(\tDE_N)\omega_b(\tDE_j)\right\}S(j,h)(\DEtInv(t)) ,
\label{eq:DE-Sinc-collocation}
\end{align}
where $\ProjDE$ is defined by~\eqref{eq:ProjDE}.
This procedure is referred to as the DE-Sinc-collocation method.

\subsection{Main result 2: Convergence of the DE-Sinc-collocation method}

In this paper,
we show the convergence of the DE-Sinc-collocation method
as follows.
The proof is provided in Section~\ref{sec:proof-DE}.

\begin{theorem}
\label{thm:DE-Sinc-collocation}
Let $\alpha$ and $d$ be positive constants with
$\alpha\leq 1$ and $d<\pi/2$.
Assume that
the assumptions on $g$ and $k$ in Theorem~\ref{thm:DE-Sinc-Nystroem}
are satisfied.
Furthermore, assume that $g\in\MC_{\alpha}(\DEt(\domD_d))$ and that
$k(\cdot,w)$ belongs to $\MC_{\alpha}(\DEt(\domD_d))$ uniformly with
respect to $w\in\DEt(\domD_d)$.
Let $h$ be selected by the formula~\eqref{eq:h-DE}.
Then, there exists a positive integer $N_0$
such that for all $N\geq N_0$,
the coefficient matrix $(I_n - \kDEn)$ is invertible.
Furthermore, there exists a constant $C$ independent of $N$
such that for all $N\geq N_0$,
\[
 \|u - \vDEn\|_{C([a,b])}
\leq C \rme^{-\pi d N/\log(2 d N/\alpha)}.
\]
\end{theorem}

Compared to Theorems~\ref{thm:SE-Sinc-collocation}
and~\ref{thm:RZ-Sinc-collocation},
we see that the convergence rate given by this theorem
is significantly improved.

As in the SE case,
the condition number of the coefficient matrix $(I_n-\kDEn)$
in the infinity norm is known to be uniformly bounded with respect to $N$
under the assumptions for the DE-Sinc-Nystr\"{o}m
method~\cite[Theorem~3.8]{okayama13:_theo}.
Because the proposed DE-Sinc-collocation method uses exactly the same
coefficient matrix, it inherits this uniform well-conditioning.

\section{Numerical experiments}
\label{sec:numer-result}

This section compares the following five methods:
the SE- and DE-Sinc-Nystr\"{o}m methods of
Muhammad et al.~\cite{muhammad05:_numer},
the SE-Sinc-collocation methods of Stenger~\cite{stenger93:_numer}
and Rashidinia--Zarebnia~\cite{rashidinia07:_solut},
and the proposed DE-Sinc-collocation method.
All experiments were performed on a MacBook Air equipped with a
1.7~GHz Intel Core i7 processor and 8~GB of memory, running macOS Big Sur.
The programs were written in C using double-precision floating-point
arithmetic and compiled with Apple clang version 13.0.0 without optimization.
The Cephes Math Library was used to evaluate the sine integral, and LAPACK
provided by Apple's Accelerate framework was used to solve the dense linear
systems.
The source code for all programs is available at
\url{https://github.com/okayamat/sinc-colloc-volterra}.

We consider the following three equations from the literature:
Example~4 of Rashidinia--Zarebnia~\cite{rashidinia07:_solut},
Equation~2.1.45 of Polyanin--Manzhirov~\cite{polyanin08:_handb},
and Example~5.1 of Fermo--van der Mee~\cite{fermo21:_volterra}:
\begin{align}
\label{eq:example1}
u(t) + \int_0^t t s u(s)\diff s
 &= \rme^{-t^2} +\frac{t}{2}(1 - \rme^{-t^2}),\quad 0\leq t\leq 1,\\
\label{eq:example2}
u(t) - 6\int_0^t (\sqrt{t} - \sqrt{s}) u(s)\diff s
 &= 1 +\sqrt{t} - 2t\sqrt{t} - t^2,\quad 0\leq t\leq 1,\\
 \label{eq:example3}
u(t) + \int_{-1}^t \cos(\omega(s-t)) u(s)\diff s
 &= \rme^t + \frac{\rme^{t+1}-\cos(\omega(t+1))+\omega\sin(\omega(t+1))}{\rme(1+\omega^2)},\quad -1\leq t\leq 1,
\end{align}
whose exact solutions are $u(t)=\rme^{-t^2}$, $u(t)=1 + \sqrt{t}$,
and $u(t)=\rme^t$, respectively.
For~\eqref{eq:example1}, the assumptions of
Theorems~\ref{thm:SE-Sinc-Nystroem}, \ref{thm:SE-Sinc-collocation}
and~\ref{thm:RZ-Sinc-collocation}
are fulfilled with $d=3.14$ and $\alpha=1$,
and those of Theorems~\ref{thm:DE-Sinc-Nystroem}
and~\ref{thm:DE-Sinc-collocation}
are fulfilled with $d=1.57$ and $\alpha=1$.
For~\eqref{eq:example2}, the assumptions of
Theorems~\ref{thm:SE-Sinc-Nystroem}, \ref{thm:SE-Sinc-collocation}
and~\ref{thm:RZ-Sinc-collocation}
are fulfilled with $d=3.14$ and $\alpha=1/2$,
and those of Theorems~\ref{thm:DE-Sinc-Nystroem}
and~\ref{thm:DE-Sinc-collocation}
are fulfilled with $d=1.57$ and $\alpha=1/2$.
For~\eqref{eq:example3}, the assumptions of
Theorems~\ref{thm:SE-Sinc-Nystroem}, \ref{thm:SE-Sinc-collocation}
and~\ref{thm:RZ-Sinc-collocation}
are fulfilled with $d=3.14$ and $\alpha=1$,
and those of Theorems~\ref{thm:DE-Sinc-Nystroem}
and~\ref{thm:DE-Sinc-collocation}
are fulfilled with $d=1.57$ and $\alpha=1$.
The mesh size $h$ was determined from~\eqref{eq:h-SE} or~\eqref{eq:h-DE}
using these values.  For the primary results for~\eqref{eq:example3},
we set $\omega=20$; the influence of $\omega$ is examined separately below.

For each method and each value of $N$, the maximum error was approximated by
\[
 E_N=\max_{1\leq \ell<M}
 \left|u(t_\ell)-u_N(t_\ell)\right|,
 \qquad
 t_\ell=a+\frac{\ell}{M}(b-a),
 \qquad M=2048,
\]
where $u_N$ denotes the corresponding approximate solution.
Thus, the same 2047 equally spaced interior test points, excluding both
endpoints, were used for all methods.  The reported computation time includes the
construction of the coefficient matrix and right-hand side, the solution of
the linear system, and the evaluation of the approximate solution at these
test points.

Stenger's SE-Sinc-collocation method and the SE-Sinc-Nystr\"{o}m method
use exactly the same linear system~\eqref{eq:SE-Sinc-Nystroem-linear-eq}.
Likewise, the proposed DE-Sinc-collocation method and the
DE-Sinc-Nystr\"{o}m method use exactly the same linear
system~\eqref{eq:DE-Sinc-Nystroem-linear-eq}.
Therefore, the costs of constructing and solving the linear systems are
identical within each pair.  Consequently, differences in the reported total
times within each pair arise from the evaluation of the approximate solutions
at the test points.
The Nystr\"{o}m solutions require evaluations of the kernel, the right-hand
side, and the sine integral.  By contrast, Stenger's method and the proposed
method are evaluated using the generalized SE- and DE-Sinc approximations,
respectively, which involve only elementary-function evaluations.
The Rashidinia--Zarebnia method uses a different linear-system formulation,
but its approximate solution is likewise based on the generalized SE-Sinc
approximation.  Therefore, its evaluation at the test points is also
inexpensive, as in the other collocation methods.
For $n=2N+1$, each dense linear system requires $\Order(n^2)$ storage and
$\Order(n^3)$ solution cost, while evaluation at $M$ test points requires
$\Order(Mn)$ operations.

Figures~\ref{fig:example1_N}--\ref{fig:example3_t} show that the two
SE-Sinc-collocation methods yield nearly identical errors and computational
performance.  This observation is consistent with
Theorems~\ref{thm:SE-Sinc-collocation} and~\ref{thm:RZ-Sinc-collocation}.
For a fixed value of $N$, the SE- and DE-Sinc-Nystr\"{o}m methods generally
produce slightly smaller errors than the corresponding collocation methods,
consistent with the sharper asymptotic estimates in
Theorems~\ref{thm:SE-Sinc-Nystroem}, \ref{thm:DE-Sinc-Nystroem},
\ref{thm:SE-Sinc-collocation}, \ref{thm:RZ-Sinc-collocation}
and~\ref{thm:DE-Sinc-collocation}.
By contrast, Figures~\ref{fig:example1_t}, \ref{fig:example2_t},
and~\ref{fig:example3_t} demonstrate that the collocation methods are more
efficient when accuracy is compared against computation time.  As explained
above, the additional evaluation cost of the Nystr\"{o}m solutions comes from
repeated evaluations of $k$, $g$, and the sine integral.

The convergence in Figure~\ref{fig:example2_N} is slower than that in
Figure~\ref{fig:example1_N}, reflecting the lower endpoint-regularity
parameter $\alpha=1/2$.  Figure~\ref{fig:example3_N} further shows that the
oscillatory kernel requires larger values of $N$ to resolve its oscillations.
Nevertheless, for the same $N$, the DE-Sinc methods remain substantially more
accurate than the SE-Sinc methods in this example.

For a quantitative comparison,
Tables~\ref{tab:example1-errors}--\ref{tab:example3-errors}
list the maximum errors at representative values of $N$.
These results confirm that the two SE-Sinc-collocation methods have
nearly identical accuracy.  They also show the almost-exponential convergence
of the DE-Sinc methods and their advantage over the SE-Sinc methods for all
three problems.

\begin{table}[htbp]
\centering
\caption{Maximum errors for~\eqref{eq:example1}.}
\label{tab:example1-errors}
\scriptsize
\setlength{\tabcolsep}{3.5pt}
\begin{tabular}{c c c c c c}
\hline
$N$
& \shortstack{SE-collocation (RZ)}
& \shortstack{SE-collocation (Stenger)}
& \shortstack{DE-collocation}
& \shortstack{SE-Nystr\"{o}m}
& \shortstack{DE-Nystr\"{o}m} \\
\hline
$10$
& $2.401\times10^{-4}$
& $2.411\times10^{-4}$
& $2.605\times10^{-4}$
& $7.157\times10^{-5}$
& $7.640\times10^{-5}$ \\
$20$
& $9.968\times10^{-6}$
& $9.975\times10^{-6}$
& $8.098\times10^{-7}$
& $2.916\times10^{-6}$
& $2.902\times10^{-7}$ \\
$40$
& $9.128\times10^{-8}$
& $9.130\times10^{-8}$
& $1.534\times10^{-11}$
& $2.926\times10^{-8}$
& $7.314\times10^{-12}$ \\
$80$
& $1.018\times10^{-10}$
& $1.018\times10^{-10}$
& $5.551\times10^{-16}$
& $3.702\times10^{-11}$
& $1.665\times10^{-16}$ \\
\hline
\end{tabular}
\end{table}

\begin{table}[htbp]
\centering
\caption{Maximum errors for~\eqref{eq:example2}.}
\label{tab:example2-errors}
\scriptsize
\setlength{\tabcolsep}{3.5pt}
\begin{tabular}{c c c c c c}
\hline
$N$
& \shortstack{SE-collocation (RZ)}
& \shortstack{SE-collocation (Stenger)}
& \shortstack{DE-collocation}
& \shortstack{SE-Nystr\"{o}m}
& \shortstack{DE-Nystr\"{o}m} \\
\hline
$10$
& $1.062\times10^{-2}$
& $1.062\times10^{-2}$
& $2.196\times10^{-3}$
& $9.094\times10^{-4}$
& $3.218\times10^{-4}$ \\
$20$
& $7.686\times10^{-4}$
& $7.686\times10^{-4}$
& $1.897\times10^{-6}$
& $1.971\times10^{-5}$
& $9.602\times10^{-8}$ \\
$40$
& $1.690\times10^{-5}$
& $1.690\times10^{-5}$
& $2.106\times10^{-12}$
& $7.737\times10^{-8}$
& $1.887\times10^{-14}$ \\
$80$
& $6.808\times10^{-8}$
& $6.808\times10^{-8}$
& $1.332\times10^{-15}$
& $2.741\times10^{-11}$
& $1.554\times10^{-15}$ \\
\hline
\end{tabular}
\end{table}

\begin{table}[htbp]
\centering
\caption{Maximum errors for~\eqref{eq:example3} with $\omega=20$.}
\label{tab:example3-errors}
\scriptsize
\setlength{\tabcolsep}{3.5pt}
\begin{tabular}{c c c c c c}
\hline
$N$
& \shortstack{SE-collocation (RZ)}
& \shortstack{SE-collocation (Stenger)}
& \shortstack{DE-collocation}
& \shortstack{SE-Nystr\"{o}m}
& \shortstack{DE-Nystr\"{o}m} \\
\hline
$20$
& $9.404\times10^{-1}$
& $9.404\times10^{-1}$
& $7.937\times10^{-1}$
& $8.809\times10^{-1}$
& $8.547\times10^{-1}$ \\
$50$
& $5.023\times10^{-2}$
& $5.023\times10^{-2}$
& $3.078\times10^{-2}$
& $1.548\times10^{-1}$
& $3.227\times10^{-2}$ \\
$100$
& $2.937\times10^{-2}$
& $2.937\times10^{-2}$
& $1.317\times10^{-6}$
& $2.899\times10^{-2}$
& $1.141\times10^{-6}$ \\
$150$
& $4.292\times10^{-3}$
& $4.292\times10^{-3}$
& $1.475\times10^{-12}$
& $4.045\times10^{-3}$
& $1.431\times10^{-12}$ \\
\hline
\end{tabular}
\end{table}

To examine the influence of the oscillation frequency more directly,
Table~\ref{tab:example3-frequency} lists the maximum errors for
\eqref{eq:example3} at $N=100$ with several values of $\omega$.
The accuracy deteriorates as $\omega$ increases because increasingly fine
sampling is required to resolve the oscillatory kernel.
The two SE-Sinc-collocation methods again produce essentially identical
errors, whereas the DE-Sinc methods maintain substantially higher accuracy
over the tested frequency range.

\begin{table}[htbp]
\centering
\caption{Maximum errors for~\eqref{eq:example3} at $N=100$.}
\label{tab:example3-frequency}
\scriptsize
\setlength{\tabcolsep}{3.5pt}
\begin{tabular}{c c c c c c}
\hline
$\omega$
& \shortstack{SE-collocation (RZ)}
& \shortstack{SE-collocation (Stenger)}
& \shortstack{DE-collocation}
& \shortstack{SE-Nystr\"{o}m}
& \shortstack{DE-Nystr\"{o}m} \\
\hline
$5$
& $9.459\times10^{-8}$
& $9.459\times10^{-8}$
& $1.776\times10^{-15}$
& $9.119\times10^{-8}$
& $8.882\times10^{-16}$ \\
$10$
& $6.196\times10^{-5}$
& $6.196\times10^{-5}$
& $1.192\times10^{-12}$
& $4.525\times10^{-5}$
& $1.015\times10^{-12}$ \\
$20$
& $2.937\times10^{-2}$
& $2.937\times10^{-2}$
& $1.317\times10^{-6}$
& $2.899\times10^{-2}$
& $1.141\times10^{-6}$ \\
$40$
& $5.195\times10^{-1}$
& $5.195\times10^{-1}$
& $2.052\times10^{-2}$
& $6.100\times10^{-1}$
& $4.288\times10^{-2}$ \\
\hline
\end{tabular}
\end{table}

\begin{figure}[htbp]
  \centering
 \includegraphics[scale=.8]{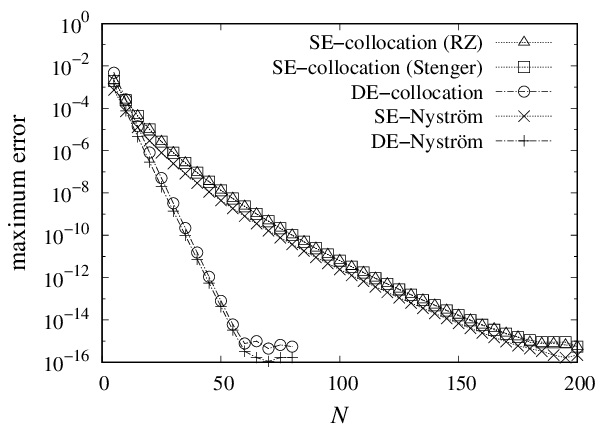}
  \caption{Maximum errors versus $N$ for~\eqref{eq:example1}.}
  \label{fig:example1_N}
\end{figure}
\begin{figure}[htbp]
  \centering
  \includegraphics[scale=.8]{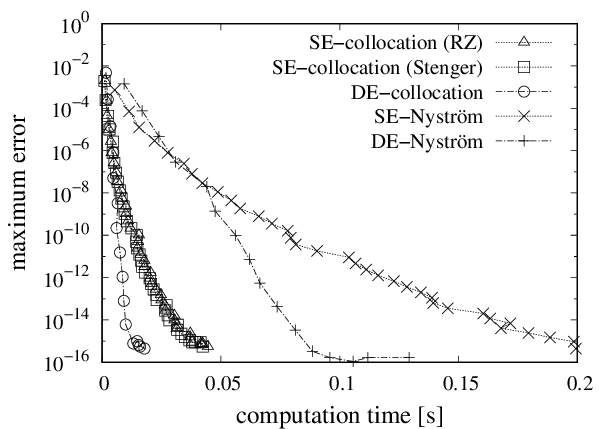}
  \caption{Maximum errors versus computation time for~\eqref{eq:example1}.}
  \label{fig:example1_t}
\end{figure}
\begin{figure}[htbp]
  \centering
 \includegraphics[scale=.8]{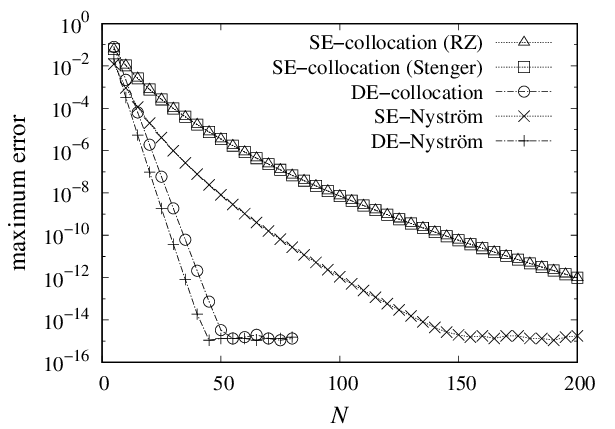}
  \caption{Maximum errors versus $N$ for~\eqref{eq:example2}.}
  \label{fig:example2_N}
\end{figure}
\begin{figure}[htbp]
  \centering
  \includegraphics[scale=.8]{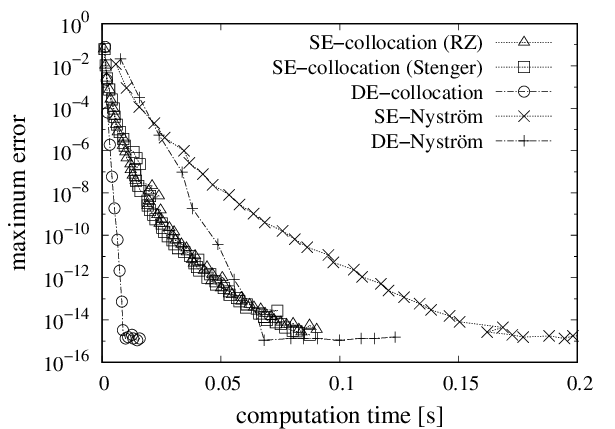}
  \caption{Maximum errors versus computation time for~\eqref{eq:example2}.}
  \label{fig:example2_t}
\end{figure}
\begin{figure}[htbp]
  \centering
 \includegraphics[scale=.8]{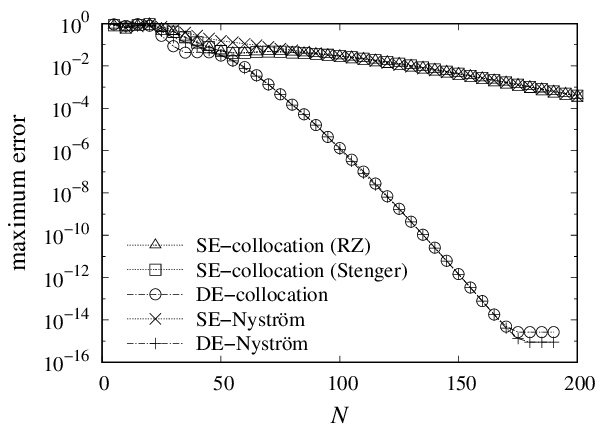}
  \caption{Maximum errors versus $N$ for~\eqref{eq:example3} with $\omega=20$.}
  \label{fig:example3_N}
\end{figure}
\begin{figure}[htbp]
  \centering
  \includegraphics[scale=.8]{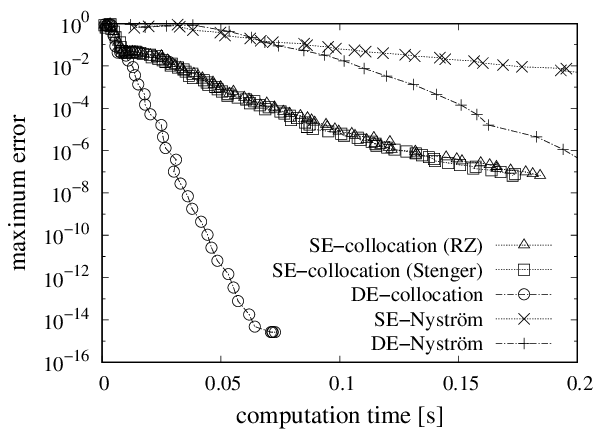}
  \caption{Maximum errors versus computation time for~\eqref{eq:example3}
  with $\omega=20$.}
  \label{fig:example3_t}
\end{figure}

\section{Proofs for the theorems presented in Section~\ref{sec:collocation}}
\label{sec:proof-SE}

In this section,
we provide proofs for Theorems~\ref{thm:equivalence},
\ref{thm:SE-Sinc-collocation}, and~\ref{thm:RZ-Sinc-collocation}.

\subsection{Proof of Theorem~\ref{thm:equivalence}}
\label{sec:proof-equivalence}

In addition to the given equation $(\Ident - \Vol)u = g$,
let us consider the following three equations:
\begin{align}
 (\Ident - \VolSEn)    \uSEn &= g, \label{eq:SE-Sinc-Nystroem-symbol} \\
 (\Ident - \ProjSE \VolSEn)v &= \ProjSE g,
 \label{eq:SE-Sinc-collocation-symbol}\\
 (\Ident - \ProjRZ \VolSEn)w &= \ProjRZ g,
\label{eq:RZ-Sinc-collocation-symbol}
\end{align}
where $\VolSEn$ and $\ProjSE$ are defined
by~\eqref{eq:VolSEn} and~\eqref{eq:ProjSE}, respectively,
and $\ProjRZ$ is defined by
\begin{align}
\label{eq:ProjRZ}
  \ProjRZ[f](t)
&=\sum_{j=-N+1}^{N-1}\left\{f(\tSE_j) - \beta_N \omega_a(\tSE_j)
 - \gamma_N \omega_b(\tSE_j)\right\}S(j,h)(\SEtInv(t))
+ \beta_N \omega_a(t) + \gamma_N \omega_b(t),
\end{align}
where $\beta_N$ and $\gamma_N$ are defined by
\begin{align*}
 \beta_N &=
\frac{f(\tSE_{-N})\omega_b(\tSE_{N}) - f(\tSE_{N})\omega_b(\tSE_{-N})}
     {\omega_a(\tSE_{-N})\omega_b(\tSE_{N}) - \omega_b(\tSE_{-N})\omega_a(\tSE_{N})},\\
\gamma_N &=
\frac{f(\tSE_{N})\omega_a(\tSE_{-N}) - f(\tSE_{-N})\omega_a(\tSE_{N})}
     {\omega_a(\tSE_{-N})\omega_b(\tSE_{N}) - \omega_b(\tSE_{-N})\omega_a(\tSE_{N})}.
\end{align*}
\begin{remark}
The denominator of $\beta_N$ and $\gamma_N$ is not zero
because
\begin{align*}
 \omega_a(\tSE_{-N})\omega_b(\tSE_{N}) - \omega_b(\tSE_{-N})\omega_a(\tSE_{N})
&=(1 - \omega_b(\tSE_{-N}))\omega_b(\tSE_N)
- \omega_b(\tSE_{-N})(1 - \omega_b(\tSE_N))\\
&=\omega_b(\tSE_N) - \omega_b(\tSE_{-N})\\
&= \tanh\left(\frac{Nh}{2}\right)\neq 0,
\end{align*}
provided that $N$ is a positive integer and $h> 0$.
\end{remark}

The essential point in the construction of $\ProjRZ$ is that the Sinc sum
is taken only over the interior indices $j=-N+1,\,\ldots,\,N-1$, whereas
$\beta_N$ and $\gamma_N$ are chosen to recover the values at the two endpoint
collocation points.  Consequently,
$\ProjRZ[f](\tSE_i)=f(\tSE_i)$ holds for $i=-N,\,\ldots,\,N$.
Stenger's method is represented using the generalized SE-Sinc approximation
operator $\ProjSE$, whereas the introduction of the distinct operator
$\ProjRZ$ enables the Rashidinia--Zarebnia method to be written in the
operator form~\eqref{eq:RZ-Sinc-collocation-symbol}.  Once this operator has
been constructed, its equivalence to the linear system~\eqref{eq:RZ-linear-eq}
follows from a standard collocation argument.

Because~\eqref{eq:SE-Sinc-Nystroem-symbol} is equivalent
to~\eqref{eq:SE-Sinc-Nystroem},
the solution of~\eqref{eq:SE-Sinc-Nystroem-symbol}
is the approximate solution of the SE-Sinc-Nystr\"{o}m method.
On~\eqref{eq:SE-Sinc-Nystroem-symbol}, the following result was obtained.

\begin{lemma}[Okayama et al.~{\cite[Lemma~6.7]{okayama13:_theo}}]
\label{lem:SE-Sinc-Nystroem}
Assume that the assumptions of Theorem~\ref{thm:SE-Sinc-Nystroem}
are satisfied.
Then, there exists a positive integer $N_0$ such that for all $N\geq N_0$,
\eqref{eq:SE-Sinc-Nystroem-symbol}
has a unique solution $\uSEn\in C([a,b])$.
Furthermore, there exists a constant $C$ independent of $N$ such that
for all $N\geq N_0$,
\begin{equation}
\label{eq:SE-Sinc-Nystroem-error}
 \|u - \uSEn\|_{C([a,b])}\leq C \|\Vol u - \VolSEn u\|_{C([a,b])}.
\end{equation}
\end{lemma}

This lemma says that~\eqref{eq:SE-Sinc-Nystroem-symbol}
has a unique solution for all sufficiently large $N$.
Using this result, we show the following three things:
\begin{enumerate}
 \item[(i)] If~\eqref{eq:SE-Sinc-Nystroem-symbol}
has a unique solution,
then~\eqref{eq:SE-Sinc-collocation-symbol}
has also a unique solution $v = \vSEn$.
 \item[(ii)] If~\eqref{eq:SE-Sinc-Nystroem-symbol}
has a unique solution,
then~\eqref{eq:RZ-Sinc-collocation-symbol}
has also a unique solution $w = \vRZn$.
 \item[(iii)] The two solutions $\vSEn$ and $\vRZn$ are not generally
equivalent, but at the collocation points,
 $\vSEn(\tSE_i) = \vRZn(\tSE_{i})$ $(i=-N,\,\ldots,\,N)$ holds.
\end{enumerate}
First, we show (i) as follows.

\begin{lemma}
\label{lem:Stenger-solution}
The following two statements are equivalent:
\begin{enumerate}
 \item[{\rm (A)}] Equation~\eqref{eq:SE-Sinc-Nystroem-symbol}
has a unique solution $\uSEn\in C([a,b])$.
 \item[{\rm (B)}] Equation~\eqref{eq:SE-Sinc-collocation-symbol} has
 a unique solution $v\in C([a,b])$.
\end{enumerate}
Furthermore, $v=\vSEn$ holds.
\end{lemma}
\begin{proof}
First, let us show $\mathrm{(A)} \Rightarrow \mathrm{(B)}$.
Note that $\VolSEn\ProjSE f = \VolSEn f$ holds
because
of the interpolation property $\ProjSE[f](\tSE_i)=f(\tSE_i)$
($i=-N,\,\ldots,\,N$).
Applying $\ProjSE$ on the both sides
of~\eqref{eq:SE-Sinc-Nystroem-symbol},
we have
\[
 \ProjSE\uSEn
 = \ProjSE(g + \VolSEn\uSEn)
 = \ProjSE(g + \VolSEn\ProjSE\uSEn),
\]
which is equivalent to $\vSEn = \ProjSE(g + \VolSEn \vSEn)$
(recall that $\vSEn = \ProjSE\uSEn$).
This equation implies
that~\eqref{eq:SE-Sinc-collocation-symbol} has a solution
$\vSEn\in C([a,b])$.

Next, we show the uniqueness.
Suppose that~\eqref{eq:SE-Sinc-collocation-symbol}
has another solution $\tilde{v}\in C([a,b])$.
Let us set a function $\tilde{u}$ as
$\tilde{u} = g + \VolSEn\tilde{v}$.
Because $\tilde{v}$ is a solution of~\eqref{eq:SE-Sinc-collocation-symbol},
we have
\[
 \tilde{v} = \ProjSE(g + \VolSEn\tilde{v}) = \ProjSE \tilde{u},
\]
from which it holds that
\[
 \tilde{u} = g + \VolSEn\tilde{v} = g + \VolSEn\ProjSE\tilde{u}
= g + \VolSEn\tilde{u}.
\]
This equation implies that $\tilde{u}$ is a solution
of~\eqref{eq:SE-Sinc-Nystroem-symbol}.
Because the solution of~\eqref{eq:SE-Sinc-Nystroem-symbol}
is unique, $\tilde{u} = \uSEn$ holds,
from which we have $\ProjSE\tilde{u}=\ProjSE\uSEn$.
Thus, we have $\tilde{v}=v$, which shows $(\mathrm{B})$.

The above argument is reversible,
which proves $\mathrm{(B)}\Rightarrow\mathrm{(A)}$.
Furthermore, in view of the proof above,
we see $v=\vSEn$, which is to be demonstrated.
\end{proof}

Next, for showing (ii),
we show the following result.
The proof is omitted because it goes in the same way as
that of Lemma~\ref{lem:Stenger-solution}.

\begin{lemma}
\label{lem:RZ-solution}
The following two statements are equivalent:
\begin{enumerate}
 \item[{\rm (A)}] Equation~\eqref{eq:SE-Sinc-Nystroem-symbol}
has a unique solution $\uSEn\in C([a,b])$.
 \item[{\rm (B)}] Equation~\eqref{eq:RZ-Sinc-collocation-symbol} has
 a unique solution $w\in C([a,b])$.
\end{enumerate}
Furthermore, $w=\ProjRZ\uSEn$ holds.
\end{lemma}

To show (ii) completely, it remains to verify that $w=\vRZn$.
The following standard collocation result provides this equivalence;
its proof follows Atkinson~\cite[Sect.\ 4.3]{atkinson97:_numer_solut}
and is therefore omitted.

\begin{proposition}
\label{prop:RZ-solution}
The following two statements are equivalent:
\begin{enumerate}
 \item[{\rm (A)}] Equation~\eqref{eq:RZ-Sinc-collocation-symbol} has
 a unique solution $w\in C([a,b])$.
 \item[{\rm (B)}] Equation~\eqref{eq:RZ-linear-eq} has
 a unique solution $\mathbd{c}_n\in \mathbb{R}^n$.
\end{enumerate}
Furthermore, $w=\vRZn$ holds.
\end{proposition}

From the above results (i) and (ii),
we find that $\vSEn = \ProjSE\uSEn$ and $\vRZn = \ProjRZ\uSEn$.
Using the interpolation property of $\ProjSE$ and $\ProjRZ$ as
\[
 \ProjSE[\uSEn](\tSE_i) = \uSEn(\tSE_i) = \ProjRZ[\uSEn](\tSE_i),
\quad i = -N,\,-N+1,\,\ldots,\,N,
\]
we have $\vSEn(\tSE_i)=\vRZn(\tSE_i)$.
However, we note that $\ProjSE$ and $\ProjRZ$ are not
generally equivalent. This can be observed
by the limits $t\to a$ and $t\to b$ as
\begin{align*}
 \lim_{t\to a}\ProjSE[f](t) = f(\tSE_{-N})
&\neq \beta_N =\lim_{t\to a}\ProjRZ[f](t),\\
 \lim_{t\to b}\ProjSE[f](t) = f(\tSE_{N})
&\neq \gamma_N =\lim_{t\to b}\ProjRZ[f](t).
\end{align*}
Thus, we obtain the claim of Theorem~\ref{thm:equivalence}.

\subsection{Proof of Theorem~\ref{thm:SE-Sinc-collocation}}
\label{sec:proof-SE-Sinc}

The invertibility of $(I_n - \kSEn)$
is already shown by Theorem~\ref{thm:SE-Sinc-Nystroem}.
Thus, we concentrate on the analysis of the error of $\vSEn$.
Because $\vSEn = \ProjSE \uSEn$, it holds that
\[
 u - \vSEn = u - \ProjSE\uSEn
= (u - \ProjSE u) + \ProjSE(u - \uSEn),
\]
which leads to
\begin{equation}
\label{eq:SE-Sinc-first-error}
 \|u - \vSEn\|_{C([a,b])}
\leq \|u - \ProjSE u\|_{C([a,b])}
 + \|\ProjSE\|_{\mathcal{L}(C([a,b]),C([a,b]))}
\|u - \uSEn\|_{C([a,b])}.
\end{equation}
For the first term, we show $u\in\MC_{\alpha}(\SEt(\domD_d))$,
from which we can use Theorem~\ref{thm:SE-Sinc-general}.
For the purpose, the following theorem is useful.

\begin{theorem}[Okayama et al.~{\cite[Theorem~3.2]{okayama13:_theo}}]
\label{thm:Sinc-Nyst-regularity}
Let $\domD=\SEt(\domD_d)$ or $\domD=\DEt(\domD_d)$.
Assume that $g$, $k(z,\cdot)$ and $k(\cdot,w)$ belong to $\Hinf(\domD)$
for all $z,\, w\in\domD$.
Then,~\eqref{eq:Volterra-int} has a unique solution $u\in\Hinf(\domD)$.
\end{theorem}

Using this theorem, we can show the following result.

\begin{theorem}
\label{thm:Sinc-collocation-regularity}
Let $\alpha$ be a positive constant with $\alpha\leq 1$.
Assume that the assumptions of Theorem~\ref{thm:Sinc-Nyst-regularity}
are satisfied. Furthermore, assume that $g\in\MC_{\alpha}(\domD)$ and that
$k(\cdot,w)$ belongs to $\MC_{\alpha}(\domD)$ uniformly with respect to
$w\in\domD$.
Then,~\eqref{eq:Volterra-int} has a unique solution $u\in\MC_{\alpha}(\domD)$.
\end{theorem}
\begin{proof}
According to Theorem~\ref{thm:Sinc-Nyst-regularity},~\eqref{eq:Volterra-int}
has a unique solution $u\in\Hinf(\domD)$.
Therefore, we only have to show the H\"{o}lder continuity
of $u$ at the endpoints.
For complex $z\in\domD$, the integrals below are taken along curves in
$\domD$ induced by the corresponding variable transformation.  These curves
can be chosen so that their lengths from $a$ to $z$ and from $z$ to $b$ are
bounded by $C_{\domD}|z-a|$ and $C_{\domD}|b-z|$, respectively, for a
constant $C_{\domD}$ independent of $z$.
Using $u = g + \Vol u$, we have
\begin{align*}
|u(b) - u(z)|
&=\left| \left(g(b) + \int_a^b k(b,w)u(w)\diff w\right)
- \left(g(z) + \int_a^z k(z,w)u(w)\diff w\right)\right|\\
&\leq \left|g(b) - g(z)\right|
+ \left|\int_z^b k(b,w)u(w)\diff w\right|
+ \left|\int_a^z \left\{k(b,w)- k(z,w)\right\}u(w)\diff w\right|.
\end{align*}
From the H\"{o}lder continuity of $g$,
the first term can be bounded by
$L_g|b - z|^{\alpha}$ for some constant $L_g$.
From the boundedness of $k$ and $u$,
the second term can be bounded by
$L_{k,u}|b - z|$ for some constant $L_{k,u}$.
Furthermore, from the boundedness of $\domD$ and $\alpha\leq 1$,
we have
$|b - z|=|b - z|^{1-\alpha}|b - z|^{\alpha}\leq L_{\domD}|b - z|^{\alpha}$
for some constant $L_{\domD}$.
From the uniform H\"{o}lder continuity of $k$ in its first variable and the
boundedness of $u$,
the third term can be bounded by
$\tilde{L}_{k,u}|b - z|^{\alpha}|z - a|$ for some constant $\tilde{L}_{k,u}$.
Furthermore, from the boundedness of $\domD$,
we have $|z - a|\leq \tilde{L}_{\domD}$ for some constant $\tilde{L}_{\domD}$.
Thus, there exists a constant $L$ such that
$|u(b) - u(z)|\leq L|b - z|^{\alpha}$,
which shows the H\"{o}lder continuity of $u$ at $z = b$.
The proof for the H\"{o}lder continuity at $z = a$ is omitted
because it follows the same method as that at $z = b$.
This completes the proof.
\end{proof}

From this theorem, we can use Theorem~\ref{thm:SE-Sinc-general}
for estimating the first term of~\eqref{eq:SE-Sinc-first-error} as
\[
 \|u - \ProjSE u\|_{C([a,b])}
\leq C_1 \sqrt{N} \rme^{-\sqrt{\pi d \alpha N}}
\]
for some constant $C_1$.
For the second term, we use the following bound for
the operator $\ProjSE$.

\begin{lemma}[Okayama~{\cite[Lemma~7.2]{okayama23:_theo}}]
Let $\ProjSE$ be defined by~\eqref{eq:ProjSE}.
Then, there exists a constant $C_2$ independent of $N$ such that
\[
 \|\ProjSE\|_{\mathcal{L}(C([a,b]),C([a,b]))}
\leq C_2 \log(N+1).
\]
\end{lemma}

The remaining term to be estimated in~\eqref{eq:SE-Sinc-first-error}
is $\|u - \uSEn\|_{C([a,b])}$.
According to Lemma~\ref{lem:SE-Sinc-Nystroem},
it is estimated as~\eqref{eq:SE-Sinc-Nystroem-error}.
Because $u\in\Hinf(\SEt(\domD_d))$, $u$ satisfies the assumptions
of Theorem~\ref{thm:SE-Sinc-indefinite},
from which we have
\[
 \|\Vol u - \VolSEn u\|_{C([a,b])}
\leq C_3 \rme^{-\sqrt{\pi d \alpha N}}.
\]
Thus, there exists a constant $C_4$ such that
\begin{align*}
 \|u - \vSEn\|_{C([a,b])}
&\leq C_1 \sqrt{N} \rme^{-\sqrt{\pi d \alpha N}}
+ C_2 \log(N+1) C_3 \rme^{-\sqrt{\pi d \alpha N}}
\leq C_4 \sqrt{N}\rme^{-\sqrt{\pi d \alpha N}}.
\end{align*}
This completes the proof of Theorem~\ref{thm:SE-Sinc-collocation}.

\subsection{Proof of Theorem~\ref{thm:RZ-Sinc-collocation}}
\label{sec:proof-RZ-Sinc}

For Theorem~\ref{thm:RZ-Sinc-collocation},
the invertibility of $(E_n^{\textRZ} - V_{n}^{\textRZ})$
can be shown by combining
Lemmas~\ref{lem:SE-Sinc-Nystroem}, \ref{lem:RZ-solution}
and Proposition~\ref{prop:RZ-solution}.
Thus, we concentrate on the analysis of the error of $\vRZn$.
By the triangle inequality, we have
\begin{align*}
 \|u - \vRZn\|_{C([a,b])}
&\leq \|u - \vSEn\|_{C([a,b])}
 + \|\vSEn - \vRZn\|_{C([a,b])}\\
&= \|u - \vSEn\|_{C([a,b])}
 + \|(\ProjSE - \ProjRZ)\uSEn\|_{C([a,b])}.
\end{align*}
Because the first term
is already estimated by Theorem~\ref{thm:SE-Sinc-collocation},
we estimate the second term.
For the purpose, the following lemma is essential.

\begin{lemma}
Let $\ProjSE: C([a,b])\to C([a,b])$ and
$\ProjRZ: C([a,b])\to C([a,b])$ be defined
by~\eqref{eq:ProjSE} and~\eqref{eq:ProjRZ}, respectively.
Then, there exists a constant $C$ independent of $N$ such that
\[
 \left\|\ProjSE - \ProjRZ\right\|_{\mathcal{L}(C([a,b]),C([a,b]))}
\leq \frac{C}{\rme^{Nh} - 1}\log(N+1).
\]
\end{lemma}
\begin{proof}
First, it holds for $f\in C([a,b])$ that
\begin{align*}
 \ProjSE[f](t) - \ProjRZ[f](t)
&= - \sum_{j=-N}^N\left\{
\left(f(\tSE_{-N})-\beta_N\right)\omega_a(\tSE_j) +
\left(f(\tSE_{N})-\gamma_N\right)\omega_b(\tSE_j)
\right\}S(j,h)(\SEtInv(t))\\
&\quad+\left(f(\tSE_{-N}) - \beta_N\right)\omega_a(t)
+\left(f(\tSE_N) - \gamma_N\right)\omega_b(t).
\end{align*}
Here, noting
\begin{align*}
 |f(\tSE_{-N}) - \beta_N|
&=\frac{|f(\tSE_N) - f(\tSE_{-N})|}{\rme^{Nh} - 1}
\leq \frac{2\|f\|_{C([a,b])}}{\rme^{Nh} - 1},\\
 |f(\tSE_{N}) - \gamma_N|
&=\frac{|f(\tSE_{-N}) - f(\tSE_{N})|}{\rme^{Nh} - 1}
\leq \frac{2\|f\|_{C([a,b])}}{\rme^{Nh} - 1},
\end{align*}
and using $\omega_a(t) + \omega_b(t) = 1$, we have
\begin{align*}
 \left\|\ProjSE - \ProjRZ\right\|_{\mathcal{L}(C([a,b]),C([a,b]))}
&\leq \frac{2}{\rme^{Nh} - 1}
\sup_{t\in(a,b)}\left\{
\sum_{j=-N}^N\left(\omega_a(\tSE_j) + \omega_b(\tSE_j)\right)
|S(j,h)(\SEtInv(t))|
+\omega_a(t) + \omega_b(t)
\right\}\\
&= \frac{2}{\rme^{Nh} - 1}
\left\{\sup_{t\in(a,b)}\sum_{j=-N}^N |S(j,h)(\SEtInv(t))| + 1\right\}\\
&\leq \frac{2}{\rme^{Nh} - 1}
\left\{\frac{2}{\pi}(3 + \log N) + 1\right\},
\end{align*}
where the standard bound~\cite[Problem 3.1.5 (a)]{stenger93:_numer}
is used for the last inequality.
Thus, the claim follows.
\end{proof}

From this lemma, substituting~\eqref{eq:h-SE} into $h$,
we estimate the second term as
\[
 \|(\ProjSE - \ProjRZ) \uSEn\|_{C([a,b])}
\leq \frac{C}{1 - \rme^{-\sqrt{\pi d / \alpha}}}
\log(N+1)\rme^{-\sqrt{\pi d N/\alpha}}\left\|\uSEn\right\|_{C([a,b])}.
\]
Noting $\alpha\in (0, 1]$, we obtain
$\rme^{-\sqrt{\pi d N/\alpha}}\leq \rme^{-\sqrt{\pi d \alpha N}}$.
Furthermore, $\log(N+1)\leq \sqrt{N}$ holds.
Hence, the proof is completed if
$\left\|\uSEn\right\|_{C([a,b])}$ is uniformly bounded
with respect to $N$.
This is shown by the following estimate
\[
 \left\|\uSEn\right\|_{C([a,b])}
\leq \left\|u - \uSEn\right\|_{C([a,b])}
+\left\| u\right\|_{C([a,b])}.
\]
From~\eqref{eq:SE-Sinc-Nystroem-error}
and Theorem~\ref{thm:SE-Sinc-indefinite},
we see that $\left\|u - \uSEn\right\|_{C([a,b])}$
converges to $0$ as $N\to\infty$,
and accordingly it is uniformly bounded.
We also see that $\left\| u\right\|_{C([a,b])}$
is bounded because $u$ is continuous on $[a, b]$
from the assumption (see Theorem~\ref{thm:Sinc-collocation-regularity}).
This completes the proof of Theorem~\ref{thm:RZ-Sinc-collocation}.

\section{Proofs for the theorem presented in Section~\ref{sec:de-collocation}}
\label{sec:proof-DE}

In this section, we provide proofs for Theorem~\ref{thm:DE-Sinc-collocation}.

\subsection{Existence and uniqueness of the approximated equations}

In addition to the given equation $(\Ident - \Vol)u = g$,
let us consider the following two equations:
\begin{align}
 (\Ident - \VolDEn)    \uDEn &= g, \label{eq:DE-Sinc-Nystroem-symbol} \\
 (\Ident - \ProjDE \VolDEn)v &= \ProjDE g,
 \label{eq:DE-Sinc-collocation-symbol}
\end{align}
where $\VolDEn$ is defined by
\[
 \VolDEn [f](t)
=\sum_{j=-N}^N k(t,\tDE_j)f(\tDE_j)\DEtDiv(jh)J(j,h)(\DEtInv(t)),
\]
and $\ProjDE$ is defined by~\eqref{eq:ProjDE}.
Because~\eqref{eq:DE-Sinc-Nystroem-symbol} is equivalent
to~\eqref{eq:DE-Sinc-Nystroem},
the solution of~\eqref{eq:DE-Sinc-Nystroem-symbol}
is the approximate solution of the DE-Sinc-Nystr\"{o}m method.
On~\eqref{eq:DE-Sinc-Nystroem-symbol}, the following result was obtained.

\begin{lemma}[Okayama et al.~{\cite[Lemma~6.10]{okayama13:_theo}}]
\label{lem:DE-Sinc-Nystroem}
Assume that the assumptions of Theorem~\ref{thm:DE-Sinc-Nystroem}
are satisfied.
Then, there exists a positive integer $N_0$ such that for all $N\geq N_0$,
\eqref{eq:DE-Sinc-Nystroem-symbol}
has a unique solution $\uDEn\in C([a,b])$.
Furthermore, there exists a constant $C$ independent of $N$ such that
for all $N\geq N_0$,
\begin{equation}
\label{eq:DE-Sinc-Nystroem-error}
 \|u - \uDEn\|_{C([a,b])}\leq C \|\Vol u - \VolDEn u\|_{C([a,b])}.
\end{equation}
\end{lemma}

On~\eqref{eq:DE-Sinc-collocation-symbol}, we can show the following lemma
in the same manner as Lemma~\ref{lem:Stenger-solution}
(hence, the proof is omitted).

\begin{lemma}
The following two statements are equivalent:
\begin{enumerate}
 \item[{\rm (A)}] Equation~\eqref{eq:DE-Sinc-Nystroem-symbol}
has a unique solution $\uDEn\in C([a,b])$.
 \item[{\rm (B)}] Equation~\eqref{eq:DE-Sinc-collocation-symbol} has
 a unique solution $v\in C([a,b])$.
\end{enumerate}
Furthermore, $v=\vDEn$ holds.
\end{lemma}

On the basis of the results,
we proceed to analyze the error of $\vDEn$ next.

\subsection{Proof of Theorem~\ref{thm:DE-Sinc-collocation}}

In the same manner as~\eqref{eq:SE-Sinc-first-error},
we have
\begin{equation}
\label{eq:DE-Sinc-first-error}
 \|u - \vDEn\|_{C([a,b])}
\leq \|u - \ProjDE u\|_{C([a,b])}
 + \|\ProjDE\|_{\mathcal{L}(C([a,b]),C([a,b]))}
\|u - \uDEn\|_{C([a,b])}.
\end{equation}
For the first term,
from Theorem~\ref{thm:Sinc-collocation-regularity},
we can use Theorem~\ref{thm:DE-Sinc-general} as
\[
 \|u - \ProjDE u\|_{C([a,b])}
\leq C_1 \rme^{-\pi d N/\log(2 d N/\alpha)}
\]
for some constant $C_1$.
For the second term, we use the following bound for
the operator $\ProjDE$.

\begin{lemma}[Okayama~{\cite[Lemma~7.5]{okayama23:_theo}}]
Let $\ProjDE$ be defined by~\eqref{eq:ProjDE}.
Then, there exists a constant $C_2$ independent of $N$ such that
\[
 \|\ProjDE\|_{\mathcal{L}(C([a,b]),C([a,b]))}
\leq C_2 \log(N+1).
\]
\end{lemma}

The remaining term to be estimated in~\eqref{eq:DE-Sinc-first-error}
is $\|u - \uDEn\|_{C([a,b])}$.
According to Lemma~\ref{lem:DE-Sinc-Nystroem},
it is estimated as~\eqref{eq:DE-Sinc-Nystroem-error}.
Because $u\in\Hinf(\DEt(\domD_d))$, $u$ satisfies the assumptions
of Theorem~\ref{thm:DE-Sinc-indefinite},
from which we have
\[
 \|\Vol u - \VolDEn u\|_{C([a,b])}
\leq C_3 \frac{\log(2 d N/\alpha)}{N}\rme^{-\pi d N/\log(2 d N/\alpha)}.
\]
Thus, there exists a constant $C_4$ such that
\begin{align*}
 \|u - \vDEn\|_{C([a,b])}
&\leq C_1 \rme^{-\pi d N/\log(2 d N/\alpha)}
+ C_2 \log(N+1) C_3 \frac{\log(2 d N/\alpha)}{N}
\rme^{-\pi d N/\log(2 d N/\alpha)}\\
&\leq C_4 \rme^{-\pi d N/\log(2 d N/\alpha)}.
\end{align*}
This completes the proof of Theorem~\ref{thm:DE-Sinc-collocation}.

\section{Conclusions}
\label{sec:conclusions}

This study clarified the relation between two independently developed
Sinc-collocation methods for Volterra integral equations of the second kind.
Stenger's method was reformulated and theoretically justified for general
two-variable kernels.
It was then proved that the approximate functions generated by Stenger's
method and the Rashidinia--Zarebnia method are not generally identical,
even though they coincide at every collocation point.
Furthermore, a rigorous convergence proof was established for the
Rashidinia--Zarebnia method, showing that both methods attain the same
root-exponential convergence rate.  The essential step was the introduction
of the projection operator $\ProjRZ$, which expresses the
Rashidinia--Zarebnia method in operator form and connects it to the
Sinc-Nystr\"{o}m equation.

The comparison provides a practical criterion for selecting a method.
Among the two SE-Sinc-collocation methods, Stenger's method is preferable
because it is simpler to implement, while having the same
convergence rate as the Rashidinia--Zarebnia method.
The Sinc-Nystr\"{o}m methods have slightly faster theoretical convergence
rates with respect to $N$, but evaluating their approximate solutions
requires repeated evaluations of $k$, $g$, and the sine integral and can
therefore be relatively expensive.
The proposed DE-Sinc-collocation method avoids this evaluation cost and
improves the convergence rate from root-exponential to almost exponential.
Furthermore, Stenger's SE-Sinc-collocation method and the proposed
DE-Sinc-collocation method inherit the uniform well-conditioning established
for the corresponding Sinc-Nystr\"{o}m coefficient matrices because they use
the same linear systems.  The Rashidinia--Zarebnia method uses a different
coefficient matrix, for which an analogous bound is not currently available.

The numerical results support these theoretical and computational
comparisons.  The two SE-Sinc-collocation methods exhibit nearly identical
accuracy and computational performance, whereas the collocation methods are
generally more efficient than the corresponding Sinc-Nystr\"{o}m methods when
accuracy is compared against computation time.  The oscillatory-kernel
example shows that larger values of $N$ are required as the frequency
increases, but the DE-Sinc methods maintain substantially higher accuracy
than the SE-Sinc methods over the tested frequency range.  These results
indicate that the proposed method is particularly suitable when high accuracy
and repeated evaluation of the approximate solution are required.
Table~\ref{tab:method-comparison} summarizes these characteristics.

\begin{table}[htbp]
\centering
\caption{Summary of the Sinc-based methods considered in this study.}
\label{tab:method-comparison}
\small
\begin{tabular}{>{\raggedright\arraybackslash}p{0.21\textwidth}
                >{\raggedright\arraybackslash}p{0.12\textwidth}
                >{\raggedright\arraybackslash}p{0.16\textwidth}
                >{\raggedright\arraybackslash}p{0.38\textwidth}}
\hline
Method & Transformation & Convergence & Main computational characteristic \\
\hline
Sinc-Nystr\"{o}m
& SE or DE
& Root- or almost exponential
& Slightly faster convergence, but relatively costly solution evaluation \\
Stenger's Sinc-collocation
& SE
& Root-exponential
& Simple formulation and inexpensive evaluation \\
Rashidinia--Zarebnia
& SE
& Root-exponential
& Same nodal values and rate as Stenger's method, but a more involved formulation \\
Proposed Sinc-collocation
& DE
& Almost exponential
& Inexpensive evaluation and high accuracy \\
\hline
\end{tabular}
\end{table}

The rigorous error estimates rely on analyticity and endpoint regularity in
appropriate transformed domains, and practical implementation requires
estimates of the regularity parameters used to determine the mesh size.
Moreover, the numerical experiments in this study were conducted for three
representative equations using double-precision arithmetic.
Further work should investigate automatic parameter selection and
roundoff-error effects for all the methods, the conditioning of the
Rashidinia--Zarebnia coefficient matrix, and performance for more challenging
problems, including higher-frequency oscillatory kernels and weakly singular
kernels.

\section*{Declaration of competing interest}

The author declares no known competing financial interests or personal
relationships that could have appeared to influence the work reported in
this paper.

\section*{Data availability}

The source code and numerical data supporting the findings of this study are
publicly available in the GitHub repository at
\url{https://github.com/okayamat/sinc-colloc-volterra}.

\section*{Declaration of generative AI and AI-assisted technologies
in the manuscript preparation process}

During the preparation of this work, the author used OpenAI Codex to
improve the language and readability and to assist with the organization
and consistency checking of the manuscript. After using this tool, the
author reviewed and edited the content as needed and takes full
responsibility for the content of the publication.

\bibliography{SincCollocationVolterra2nd}

\end{document}